\documentclass[11pt]{article}
\usepackage{etex}
\usepackage[all]{xy}
\usepackage{cite}
\usepackage[utf8]{inputenc}
\usepackage{url}
\usepackage{todonotes}
\usepackage{amsfonts}
\usepackage{amsthm}
\usepackage{enumerate}
\usepackage{graphicx}
\usepackage{mathrsfs}
\usepackage{bm}
\usepackage{cite}
\usepackage[colorlinks,linkcolor=blue,citecolor=blue]{hyperref}
\usepackage{amssymb,amsmath} % font used for R in Real numbers
\usepackage{makecell}
\usepackage{tikz}
\usepackage{pgf}
\usepackage{tikz}
\usetikzlibrary{patterns}
\usepackage{pgffor}
\usepackage{pgfcalendar}
\usepackage{pgfpages}
\usepackage{shuffle,yfonts}
\usepackage{mathtools}
\DeclareFontFamily{U}{shuffle}{}
\DeclareFontShape{U}{shuffle}{m}{n}{ <-8>shuffle7 <8->shuffle10}{}

\newcommand{\nc}{\newcommand}
\nc{\AMZV}{\mathsf {AMZV}}
\nc{\ud}{\mathrm{d}}
\nc{\ES}{\mathsf {ES}}
\nc{\MZV}{\mathsf {MZV}}
\nc{\MtV}{\mathsf {MtV}}
\nc{\MTV}{\mathsf {MTV}}
\nc{\MSV}{\mathsf {MSV}}
\nc{\MMV}{\mathsf {MMV}}
\nc{\MMVo}{\mathsf {MMVo}}
\nc{\MMVe}{\mathsf {MMVe}}
\nc{\AMMV}{\mathsf {AMMV}}
\nc{\AMTV}{\mathsf {AMTV}}
\nc{\AMtV}{\mathsf {AMtV}}
\nc{\AMSV}{\mathsf {AMSV}}
\nc{\CMZV}{\mathsf {CMZV}}
\nc{\sha}{\shuffle}
\nc{\cst}{\rotatebox[origin=c]{180}{$\sha$}}
\nc{\cstt}{\rotatebox[origin=c]{180}{$\scriptstyle \sha$}}
\nc{\de}{\delta}
\nc{\DD}{{\mathbb D}}
\nc{\anbb}[1]{\left\langle#1\right\rangle}
\nc{\bibb}[1]{\left\{#1\right\}}
\nc{\mibb}[1]{\left[#1\right]}
\nc{\smbb}[1]{\left(#1\right)}
\nc{\doubb}[1]{\llbracket#1\rrbracket}
\nc{\dm}[1]{\left|#1\right|}

\nc{\Gbinom}[2]{\genfrac{(}{)}{0mm}{0}{#1}{#2}}
\nc{\gbinom}[2]{\genfrac{(}{)}{0mm}{1}{#1}{#2}}
\nc{\Rbinom}[2]{\genfrac{\langle}{\rangle}{0mm}{0}{#1}{#2}}
\nc{\rbinom}[2]{\genfrac{\langle}{\rangle}{0mm}{1}{#1}{#2}}
\nc{\Qbinom}[2]{\genfrac{[}{]}{0mm}{0}{#1}{#2}_q}
\nc{\qbinom}[2]{\genfrac{[}{]}{0mm}{1}{#1}{#2}_q}

\nc{\binq}[2]{\genfrac{[}{]}{0mm}{0}{#1}{#2}}
\nc{\tbnq}[2]{\genfrac{[}{]}{0mm}{1}{#1}{#2}}
\nc{\cinq}[2]{\genfrac{\{}{\}}{0mm}{0}{#1}{#2}}
\nc{\tcnq}[2]{\genfrac{\{}{\}}{0mm}{1}{#1}{#2}}

\nc{\mfrac}[2]{\genfrac{}{}{0pt}{}{#1}{#2}}
\nc{\tf}{\tfrac}
\nc{\db}{{\mathbb D}}
\nc{\pari}{{\rm par}}
\nc{\dk}{{\mathbb K}}
\nc{\ola}{\overleftarrow}
\nc{\ora}{\overrightarrow}
\nc{\lra}{\longrightarrow}
\nc{\Lra}{\Longrightarrow}
\nc\Res{{\rm Res}}
\nc\setX{{\mathsf{X}}}
\nc\fA{{\mathfrak{A}}}
\nc\evaM{{\texttt{M}}}
\nc\evaML{{\text{\em{\texttt{M}}}}}
\nc\z{{\texttt{z}}}
\nc\emz{\emph{\texttt{z}}}
\nc\tx{{\texttt{x}}}
\nc\txp{{\tx_1}} % textstyle x positive 1
\nc\txn{{\tx_{-1}}} % textstyle x negative 1
\nc\neo{{1}}
\nc{\yi}{{1}}
\nc\one{{-1}}
\nc\gD{{\Delta}}
\nc\eps{{\varepsilon}}
\nc{\bfMB}{{\bf MB}}
\nc{\bftB}{{\bf tB}}
\nc{\bfTB}{{\bf TB}}
\nc{\bfSB}{{\bf SB}}
\nc{\bfB}{{\bf B}}
\nc{\bfp}{{\bf p}}
\nc{\bfq}{{\bf q}}
\nc{\bfr}{{\bf r}}
\nc{\bfu}{{\bf u}}
\nc{\bfv}{{\bf v}}
\nc{\bfa}{{\bf a}}
\nc{\bfw}{{\bf w}}
\nc{\bfy}{{\bf y}}
\nc{\T}{\ddot{t}}
\nc{\bfe}{{\boldsymbol{\sl{e}}}}
\nc{\bfi}{{\boldsymbol{\sl{i}}}}
\nc{\bfj}{{\boldsymbol{\sl{j}}}}
\nc{\bfk}{{\boldsymbol{\sl{k}}}}
\nc{\bfl}{{\boldsymbol{\sl{l}}}}
\nc{\bfm}{{\boldsymbol{\sl{m}}}}
\nc{\bfn}{{\boldsymbol{\sl{n}}}}
\nc{\bfs}{{\boldsymbol{\sl{s}}}}
\nc{\bft}{{\boldsymbol{\sl{t}}}}
\nc{\bfx}{{\boldsymbol{\sl{x}}}}
\nc{\bfz}{{\boldsymbol{\sl{z}}}}
\nc\bfgs{{\boldsymbol \gs}}
\nc\bfgl{{\boldsymbol \lambda}}
\nc\bfsi{{\boldsymbol \gs}}
\nc\bfet{{\boldsymbol \eta}}
\nc\bfeta{{\boldsymbol \eta}}
\nc\bfeps{{\boldsymbol \eps}}
\nc\mmu{{\boldsymbol \mu}}
\nc\bfone{{\bf 1}}
\nc{\myone}{{1}}

 \nc{\calA}{{\mathcal A}}
 \nc{\calB}{{\mathcal B}}
 \nc{\calC}{{\mathcal C}}
 \nc{\calD}{{\mathcal D}}
 \nc{\calE}{{\mathcal E}}
 \nc{\calF}{{\mathcal F}}
 \nc{\calG}{{\mathcal G}}
 \nc{\calH}{{\mathcal H}}
 \nc{\calI}{{\mathcal I}}
 \nc{\calJ}{{\mathcal J}}
 \nc{\calK}{{\mathcal K}}
 \nc{\calL}{{\mathcal L}}
 \nc{\calM}{{\mathcal M}}
 \nc{\calN}{{\mathcal N}}
 \nc{\calO}{{\mathcal O}}
 \nc{\calP}{{\mathcal P}}
 \nc{\calQ}{{\mathcal Q}}
 \nc{\calR}{{\mathcal R}}
 \nc{\calS}{{\mathcal S}}
 \nc{\calT}{{\mathcal T}}
 \nc{\calU}{{\mathcal U}}
 \nc{\calV}{{\mathcal V}}
 \nc{\calW}{{\mathcal W}}
 \nc{\calX}{{\mathcal X}}
 \nc{\calY}{{\mathcal Y}}
 \nc{\calZ}{{\mathcal Z}}
  \nc{\cala}{{\mathcal a}}
 \nc{\calb}{{\mathcal b}}
 \nc{\calc}{{\mathcal c}}
 \nc{\cald}{{\mathcal d}}
 \nc{\cale}{{\mathcal e}}
 \nc{\calf}{{\mathcal f}}
 \nc{\calg}{{\mathcal g}}
 \nc{\calh}{{\mathcal h}}
 \nc{\cali}{{\mathcal i}}
 \nc{\calj}{{\mathcal j}}
 \nc{\calk}{{\mathcal k}}
 \nc{\call}{{\mathcal l}}
 \nc{\calm}{{\mathcal m}}
 \nc{\caln}{{\mathcal n}}
 \nc{\calo}{{\mathcal o}}
 \nc{\calp}{{\mathsf p}}
 \nc{\calq}{{\mathcal q}}
 \nc{\calr}{{\mathcal r}}
 \nc{\cals}{{\mathcal s}}
 \nc{\calt}{{\mathcal t}}
 \nc{\calu}{{\mathcal u}}
 \nc{\calv}{{\mathcal v}}
 \nc{\calw}{{\mathcal w}}
 \nc{\calx}{{\mathcal x}}
 \nc{\caly}{{\mathcal y}}
 \nc{\calz}{{\mathcal z}}
 \nc{\ot}{{\otimes}}
\usetikzlibrary{arrows,shapes,chains}
 \allowdisplaybreaks

\catcode`!=11
\let\!int\int \def\int{\displaystyle\!int}
\let\!lim\lim \def\lim{\displaystyle\!lim}
\let\!sum\sum \def\sum{\displaystyle\!sum}
\let\!sup\sup \def\sup{\displaystyle\!sup}
\let\!inf\inf \def\inf{\displaystyle\!inf}
\let\!cap\cap \def\cap{\displaystyle\!cap}
\let\!max\max \def\max{\displaystyle\!max}
\let\!min\min \def\min{\displaystyle\!min}
\let\!frac\frac \def\frac{\displaystyle\!frac}
\catcode`!=12

\nc{\gam}{{\gamma}}
\nc{\gG}{{\Gamma}}
\nc{\om}{{\omega}}
\nc{\vep}{{\varepsilon}}
\nc{\ga}{{\alpha}}
\nc{\gl}{{\lambda}}
\nc{\gb}{{\beta}}
\nc{\gd}{{\delta}}
\nc{\gf}{{\varphi}}
\nc{\gs}{{\sigma}}
\nc{\gk}{{\kappa}}
\nc{\gS}{\Sigma}
\let\oldsection\section
\renewcommand\section{\setcounter{equation}{0}\oldsection}

\allowdisplaybreaks

\DeclareMathOperator{\Li}{Li}

\DeclareMathOperator{\ti}{ti}

\nc\UU{\mbox{\bfseries U}}
\nc\FF{\mbox{\bfseries \itshape F}}
\nc\h{\mbox{\bfseries \itshape h}}\nc\dd{\mbox{d}}
\nc\g{\mbox{\bfseries \itshape g}}
\nc\xx{\mbox{\bfseries \itshape x}}

\def\N{\mathbb{N}}
\def\Z{\mathbb{Z}}

\def\xx{\left(\frac{1-x}{1+x} \right)}

\nc\divg{{\text{div}}}
\theoremstyle{plain}
\newtheorem{thm}{Theorem}[section]
\newtheorem{lem}[thm]{Lemma}

\newtheorem{cor}[thm]{Corollary}
\newtheorem{con}[thm]{Conjecture}

\theoremstyle{definition}

\newtheorem{re}[thm]{Remark}
\newtheorem{exa}[thm]{Example}
\newtheorem{qu}{Question}[section]

\nc{\cicc}[1]{{}_{{}^{ \bigcirc\hskip-1.2ex{#1}\hskip.3ex{}}}}
\nc{\cic}[1]{{}^{\bigcirc\hskip-1.15ex{\raisebox{-0.015cm}{\text{$\scriptscriptstyle #1$}}}\hskip.25ex{}}}
\nc{\ccic}[1]{{}^{\bigcirc\hskip-1.5ex{\raisebox{-0.015cm}{\text{$\scriptscriptstyle #1$}}}\hskip.25ex{}}}
\nc{\ncic}[1]{ {\bigcirc\hskip-1.6ex{\raisebox{-0.0cm}{\text{$\scriptstyle #1$}}}\hskip.25ex{}}}
\nc{\nncic}[1]{ {\bigcirc\hskip-2ex{\raisebox{-0.0cm}{\text{$\scriptstyle #1$}}}\hskip.25ex{}}}
\nc{\cci}[1]{{}_{{}^{ {\textstyle \bigcirc}\hskip-2.05ex{#1}\hskip-.35ex{}}}}
\nc{\ccicc}[1]{{}_{{}^{ {\textstyle \bigcirc}\hskip-1.55ex{#1}\hskip-0.1ex{}}}}
\nc{\x}{\rm{x}}
\nc{\tworow}[2]{\left(#1 \atop #2\right)}

\nc{\fl}{{\mathfrak l}}
\nc{\fm}{{\mathfrak m}}

\begin{document}
%%%%%%%%%%%%%%%%%%%% title %%%%%%%%%%%%%%%%%%%%%%%%%%%%%%%%%%%%%%%%%%%%%%%%
\title{\bf Symmetry Results for Cyclotomic Multiple Hurwitz Zeta Values via Contour Integrals}
\author{
{Ce Xu\thanks{Email: cexu2020@ahnu.edu.cn}}\\[1mm]
\small School of Mathematics and Statistics, Anhui Normal University,\\ \small Wuhu 241002, P.R. China
}
\date{}
\maketitle

\noindent{\bf Abstract.} This paper provides a systematic study of symmetry properties for cyclotomic multiple Hurwitz zeta values with multiple variables and parameters by applying the methods of contour integration and the residue theorem. The main contributions are the derivation of explicit symmetry formulas for cyclotomic multiple (Hurwitz) zeta values, which are obtained directly through analytic residue calculations, without reliance on algebraic regularization. As a concrete application, we deduce analogous symmetry theorems for cyclotomic multiple zeta values and cyclotomic multiple $t$-values. The results extend and complement recent symmetry investigations by Charlton and Hoffman, offering completely explicit and regularization-free formulas in the convergent setting. Moreover, the results of this paper can be used to prove the symmetry conjecture for cyclotomic multiple Hurwitz zeta values with multiple variables and a single parameter. Furthermore, several illustrative corollaries and examples are included, and an open problem concerning possible extensions to other variants of multiple zeta values is posed at the conclusion.

\medskip

\medskip


%\setcounter{section}{-1}
%\tableofcontents
\section{Introduction}

In a recent paper \cite{CharltonHoffman-MathZ2025}, Charlton and Hoffman investigated the symmetry theorem for regularized cyclotomic multiple $t$-values (\cite[Thms. 1.1 and 2.21]{CharltonHoffman-MathZ2025}) using stuffle regularization and generating functions of truncated cyclotomic multiple $t$-values. The motivation of this paper is to extend the symmetry theorem obtained by Charlton and Hoffman for cyclotomic multiple $t$-values to the more general setting of cyclotomic multiple Hurwitz zeta values, providing explicit formulas (see Theorems \ref{mainthm-CMZVs} and \ref{mainthm-CMHZVs}) without employing any algebraic tools-instead relying solely on analytic and purely analytical theories and methods. That is to say, we abandon the regularization process and consider only the convergent cases.

To facilitate the presentation of our main results and the methods used, we first introduce some specific notations.
Let $i$ and $j$ be integers greater than or equal to 0, with $i\leq j$. For any complex numbers $z_i, z_{i+1},\ldots, z_j$, define the following sequences:
\begin{align*}
\overrightarrow{\bfz}_{i,j}:=(z_i,z_{i+1},\ldots,z_j)\quad\text{and}\quad \overleftarrow{\bfz}_{i,j}:=(z_j,z_{j-1},\ldots,z_i),
\end{align*}
where $\overrightarrow{\bfz}_{i,j}=\overleftarrow{\bfz}_{i,j}:=\emptyset$ if $i>j$. If a superscript $-1$ is added, it is conventionally represented as:
\begin{align*}
\overrightarrow{\bfz}^{-1}_{i,j}:=(1/z_i,1/z_{i+1},\ldots,1/z_j)\quad\text{and}\quad \overleftarrow{\bfz}^{-1}_{i,j}:=(1/z_j,1/z_{j-1},\ldots,1/z_i)\quad (\forall\ z_l\in \mathbb{C}\setminus\{0\}).
\end{align*}
This paper also agrees that $z$ and $\bfz$ above can be replaced with other letters for similar definitions, such as $\overrightarrow{\bfy}_{i,j}=(y_i,y_{i+1},\ldots,y_j)$ and $\overleftarrow{\bfy}_{i,j}:=(y_j,y_{j-1},\ldots,y_i)$. In particular, for the addition of two sequences, we define
\begin{align*}
\overrightarrow{(\bfy\pm \bfz)}_{i,j}:=(y_i\pm z_i,y_{i+1}\pm z_{i+1},\ldots,y_j\pm z_j)\quad\text{and}\quad \overleftarrow{(\bfy\pm\bfz)}_{i,j}:=(y_j\pm z_j,y_{j-1}\pm z_{j-1},\ldots,y_i\pm z_i).
\end{align*}
Furthermore, if all elements in a sequence are equal, for example, $\overrightarrow{\bfy}_{i,j}=\overleftarrow{\bfy}_{i,j}=(\{y\}_{j-i})$ (where $\{y\}_r$ denotes the sequence of $y$ repeated $r$ times), then we adopt the following notation (the order of $y$ and $\bfz$ can also be interchanged)
\begin{align*}
\overrightarrow{(y\pm\bfz)}_{i,j}:=(y\pm z_i,y\pm z_{i+1},\ldots,y\pm z_j)\quad\text{and}\quad \overleftarrow{(y\pm \bfz)}_{i,j}:=(y\pm z_j,y\pm z_{j-1},\ldots,y\pm z_i).
\end{align*}
Let $i,j\in\N_0:=\N\cup\{0\}$ with $i\leq j$. If $k_i,k_{i+1},\ldots,k_j\in\N$, we define
\begin{align}
|\overrightarrow{\bfk}_{i,j}|:=k_i+k_{i+1}+\cdots+k_j\quad\text{and}\quad |\emptyset|:=0.
\end{align}

Let $\overrightarrow{\bfa}_{1,r}:=(a_1,\ldots,a_r)$ and $a_j\in \mathbb{C}\setminus \N^-\ (\N^-:=\{-1,-2,-3,\ldots\})$. For $\overrightarrow{\bfk}_{1,r}:=(k_1,\ldots,k_r)\in \N^r$ and $\overrightarrow{\bfx}_{1,r}:=(x_1,\ldots,x_r)$, where each $x_j$ is a root of unity and $(k_r,x_r)\neq (1,1)$, we define the \emph{cyclotomic multiple Hurwitz zeta values} with $r$-variables (namely $x_1,\ldots,x_r$) and $r$-parameters (namely $a_1,\ldots,a_r$) by
\begin{align}\label{defn-CMHZV}
\Li_{\overrightarrow{\bfk}_{1,r}}\left(\overrightarrow{\bfx}_{1,r};\overrightarrow{(1+\bfa)}_{1,r}\right)&\equiv \Li_{k_1,\ldots,k_r}(x_1,\ldots,x_r;a_1+1,\ldots,a_r+1)\nonumber\\
&:=\sum_{0<n_1<\cdots<n_r} \frac{x_1^{n_1}\cdots x_r^{n_r}}{(n_1+a_1)^{k_1}\cdots (n_r+a_r)^{k_r}},
\end{align}
where $r$ and $|\overrightarrow{\bfk}_{1,r}|=k_1+\cdots+k_r$ are called the \emph{depth} and \emph{weight}, respectively.
In particular, when all $x_j=1$, \eqref{defn-CMHZV} is referred to as the \emph{multiple Hurwitz zeta value} (also see \cite{AkiyamaIshikawa}). This can be viewed as a Hurwitz-type generalization of the classical \emph{multiple zeta values}, which are defined as follows (see \cite{H1992,DZ1994} ):
\begin{align}\label{defn-MZV}
\zeta\left(\overrightarrow{\bfk}_{1,r}\right)\equiv \zeta(k_1,\ldots,k_r):=\sum_{0<n_1<\cdots<n_r} \frac{1}{n_1^{k_1}\cdots n_r^{k_r}},
\end{align}
where $k_1,\ldots,k_r$ are positive integers and $k_r\geq 2$ (i.e. admissible).
If all $a_j=0$ then \eqref{defn-CMHZV} is called the \emph{cyclotomic multiple zeta value}, denoted by the symbol (see \cite{YuanZh2014a,Zhao2010})
\begin{align}\label{defn-CMZV}
\Li_{\overrightarrow{\bfk}_{1,r}}\left(\overrightarrow{\bfx}_{1,r}\right)\equiv \Li_{\overrightarrow{\bfk}_{1,r}}\left(\overrightarrow{\bfx}_{1,r};\{1\}_r\right)=\sum_{0<n_1<\cdots<n_r} \frac{x_1^{n_1}\cdots x_r^{n_r}}{n_1^{k_1}\cdots n_r^{k_r}}.
\end{align}
Furthermore, if all \(x_j \in \{\pm 1\}\) in \eqref{defn-CMZV} and at least one \(x_j = -1\), then it is referred to as an \emph{alternating multiple zeta value} (\cite{BBB1997}). It is worth emphasizing that many special alternating multiple zeta values can essentially be expressed as rational linear combinations of multiple zeta values. An alternating multiple zeta value that admits such a representation is referred to as an \emph{unramified} alternating multiple zeta value. For instance, Zhao proved in \cite{Zhao2-2010} the following beautiful unramified family of alternating multiple zeta values, which was initially discovered numerically in \cite{BBB1997}:
\[\zeta(\{1,{\bar 2}\}_l)=8^{-l}\zeta(\{3\}_l)\quad (l \in\N_0)\]
is one such unramified alternating multiple zeta value.
In fact, cyclotomic multiple zeta values are a special case of multiple polylogarithm function with $r$-variables. The \emph{multiple polylogarithm function with $r$-variables} refer to the case where all $x_j$ in \eqref{defn-CMZV} are treated as continuous variables and satisfy the condition $|x_j\cdots x_r|\leq 1$ with $(k_r,x_r)\neq (1,1)$. It can be analytically continued to a multi-valued meromorphic function on $\mathbb{C}^r$ (see \cite{Zhao2007d}). In particular, if $a_1=a_2=\cdots=a_r=a$ in \eqref{defn-CMHZV}, we adopt the notation (We refer to these as cyclotomic multiple Hurwitz zeta values with $r$-variables and single parameter):
\begin{align}\label{defn-CMHZV-single}
\Li_{\overrightarrow{\bfk}_{1,r}}\left(\overrightarrow{\bfx}_{1,r};a+1\right)&\equiv \Li_{k_1,\ldots,k_r}(x_1,\ldots,x_r;\{a+1\}_r)\nonumber\\&=\sum_{0<n_1<\cdots<n_r} \frac{x_1^{n_1}\cdots x_r^{n_r}}{(n_1+a)^{k_1}\cdots (n_r+a)^{k_r}}.
\end{align}
Recently, Rui \cite{Rui2026} proposed the following symmetry conjecture concerning cyclotomic multiple Hurwitz zeta values with $r$-variables and single parameter:
\begin{con}(\cite[Conjecture 4.6]{Rui2026})\label{conjsymetrycmHZv}
Let $x_1,\ldots,x_r$ be roots of unity and $a\in \mathbb{C}\setminus \Z$. For $\bfk=(k_1,\ldots,k_r)\in \N^r$ and $(k_1,x_1),\ (k_r,x_r)\neq (1,1)$, we have
\begin{align*}
\Li_{k_1,\ldots,k_r}(x_1,\ldots,x_r;a)\equiv  (-1)^{k_1+\cdots+k_r-1} (x_1\cdots x_r) \Li_{k_r,\ldots,k_1}(x^{-1}_r,\ldots,x_1^{-1};1-a)\quad (\text{mod products}),
\end{align*}
where ``mod products" means discarding all product terms of cyclotomic multiple Hurwitz zeta values with depth $<r$.
\end{con}
Clearly, when $a=1/2$ in the above conjecture, a symmetry result concerning cyclotomic multiple $t$-values can be obtained. When all $a_j=-1/2$ in \eqref{defn-CMHZV}, it is referred to as a \emph{cyclotomic multiple $t$-value}, denoted by the symbol $\ti_{\overrightarrow{\bfk}_{1,r}}(\overrightarrow{\bfx}_{1,r})$, that is
\begin{align}\label{defn-CMtV}
\ti_{\overrightarrow{\bfk}_{1,r}}\left(\overrightarrow{\bfx}_{1,r}\right)&\equiv \ti_{k_1,\ldots,k_r}(x_1,\ldots,x_r)=\sum_{0<n_1<\cdots<n_r} \frac{x_1^{n_1}\cdots x_r^{n_r}}{(n_1-1/2)^{k_1}\cdots (n_r-1/2)^{k_r}}.
\end{align}
This paper employs the methods of contour integration and residue computation to prove the Conjecture \ref{conjsymetrycmHZv} (see Corollary \ref{main-thm-corthree} and Remark \ref{re-conj}), and further presents a more generalized form along with explicit formulas.

The concept of multiple zeta values was independently introduced in the early 1990s by Hoffman \cite{H1992} and Zagier \cite{DZ1994}. Due to their profound connections with diverse fields such as knot theory, algebraic geometry, and theoretical physics, the study of multiple zeta values has attracted significant attention from mathematicians and physicists. Over more than three decades of development, this area has generated a wealth of research. For a comprehensive overview of results prior to 2016, readers may refer to Zhao's monograph \cite{Z2016}. The classical multiple $t$-values $t(k_1,\ldots,k_r)\equiv \ti_{k_1,\ldots,k_r}(\{1\}_r)$ were formally introduced by Hoffman in 2018 (\cite{H2019}). However, even before Hoffman's formal definition, multiple $t$-values had already been studied by several researchers. For example, Zhao \cite{Zhao2015} investigated certain sum formula properties of multiple $t$-values in 2015, and in his recent 2024 paper \cite{Z2024}, he systematically examined sum formulas for multiple zeta values of level two. Certainly, there exist many other variants of multiple zeta values, such as Kaneko-Tsumura multiple
$T$-values, Mordell-Tornheim zeta functions, motivic multiple zeta/$t-$/$T-$ values, and so on. For further details, refer to the literature \cite{Charlton-MathAnn2025,DSS2025,KanekoTs2019,KMT2023,Li2024,LiLiu2021,Murakami2021} and related references. Obviously, if setting $a_j={-j}/{2}$ and all $x_j=1$ in \eqref{defn-CMHZV} then yields the Kaneko-Tsumura's multiple $T$-values (\cite{KanekoTs2019}):
\begin{align*}
\Li_{k_1,\ldots,k_r}\left(\{1\}_r;1-\frac1{2},\ldots,1-\frac{r}{2}\right)=2^{k_1+\cdots+k_r-r} T(k_1,\ldots,k_r),
\end{align*}
where the \emph{multiple $T$-values} $T(k_1,\ldots,k_r)$ is defined by $(k_r>1)$
\begin{align*}
T(k_1,\ldots,k_r):=2^r\sum\limits_{0<n_1<\cdots<n_r} \frac{1}{(2n_1-1)^{k_1} (2n_2-2)^{k_2}\cdots(2n_r-r)^{k_r}}.
\end{align*}

In 1998, Flajolet and Salvy \cite{Flajolet-Salvy} developed a contour integral theory to systematically study \emph{generalized Euler sums} of the following general form:
\begin{align}
{S_{{p_1p_2\cdots p_k},q}} := \sum\limits_{n = 1}^\infty  {\frac{{H_n^{\left( {{p_1}} \right)}H_n^{\left( {{p_2}} \right)} \cdots H_n^{\left( {{p_k}} \right)}}}
{{{n^{q}}}}},
\end{align}
where $p_j\in \N$ and $q\geq 2$. Here $H_n^{(p)}$ stands the {\emph{generalized harmonic number}} of order $p$ defined by
\[H_n^{(p)}:=\sum_{k=1}^n \frac{1}{k^p}\quad (p\in\N).\]
According to the stuffle relations, it is straightforward to see that the generalized Euler sums mentioned above can be expressed as integer-coefficient linear combinations of multiple zeta values. The method they employed involves considering the following contour integral:
\begin{align*}
\oint_{(\infty)} r(s)\xi(s) ds:=\lim_{R\rightarrow \infty}\oint_{C_R} r(s)\xi(s) ds =0,
\end{align*}
where $C_R$ denotes a circular contour with radius $R$. Here $\xi(s)$ is referred to as a kernel function, defined as
\begin{align*}
\xi(s)=\frac{\pi\cot(\pi s)\psi^{(p_1-1)}(-s)\psi^{(p_2-1)}(-s)\cdots \psi^{(p_k-1)}(-s)}{(p_1-1)!(p_2-1)!\cdots(p_k-1)!}
\end{align*} and $r(s)$ is a basis function, defined as $r(s)=1/s^q\ (\forall p_j\in\N,\ q\in \N\setminus \{0\})$. $\psi(s)$ denotes the classical \emph{digamma function} defined by
\begin{align}\label{defn-classical-psi-funtion}
\psi(s)=-\gamma-\frac1{s}+\sum_{k=1}^\infty \left(\frac1{k}-\frac{1}{s+k}\right),
\end{align}
where $s\in\mathbb{C}\setminus \N_0^-$ and $\N_0^-:=\N^-\cup\{0\}=\{0,-1,-2,-3,\ldots\}$.

The main purpose of this paper is to systematically investigate the symmetry results for cyclotomic multiple Hurwitz zeta values with multiple variables and parameters by employing the method of contour integration and residue calculation developed by Flajolet and Salvy \cite{Flajolet-Salvy}, thereby deriving symmetry results for cyclotomic multiple $t$-values. In contrast, our primary approach to investigating the main theorems of this paper involves considering the following two contour integrals: for $q\in\N$ and roots of unity $x,x_1,\ldots,x_r$ with $(q,x)$ and $(k_r,x_r)\neq (1,1)$,
\begin{align}
\oint\limits_{\left( \infty  \right)} \frac{\Phi(s;x)\Li_{k_1,\ldots,k_r}(x_1,\ldots,x_r;s+1)}{s^q} ds=0\label{contour-integral-one}
\end{align}
and
\begin{align}
&\oint\limits_{\left( \infty  \right)} \frac{\Phi(s;x)\Li_{k_1,\ldots,k_r}(x_1,\ldots,x_r;1+a_1+s,\ldots,1+a_r+s)}{(s+a)^q} ds=0,\label{contour-integral-two}
\end{align}
where $\Phi(s;x)$ denotes the \emph{extended trigonometric function} defined through the \emph{generalized digamma function} $\phi(s;x)$ (see \cite{Rui-Xu2025})
\begin{align}\label{etrif}
\Phi(s;x):=\phi(s;x)-\phi\Big(-s;x^{-1}\Big)-\frac1{s},
\end{align}
with $\phi(s;x)$ defined as
\begin{align}\label{gdigf}
\phi(s;x):=\sum_{k=0}^\infty \frac{x^k}{k+s}\quad (s\notin\N_0^-).
\end{align}
Here, the convergence of \eqref{etrif} is guaranteed only when $x$ takes a root of unity, whereas in \eqref{gdigf}, $x$ may be any value satisfying $|x|\leq 1$ and $x\neq 1$.

Both contour integrals \eqref{contour-integral-one} and \eqref{contour-integral-two} are zero because, for a circular contour of radius \(R\to\infty\), the integrand decays sufficiently fast. Specifically, the kernel \(\Phi(s;x)\) grows at most like \(O(|s|^{-2})\) for roots of unity \(x\neq1\) (and remains bounded when \(x=1\) but then \(q\ge2\) ensures decay). The multiple polylogarithm factor is bounded in the relevant region, while the rational factors \(s^{-q}\) or \((s+a)^{-q}\) provide an extra decay of order \(|s|^{-q}\). Hence the whole integrand is \(O(|s|^{-2-q})\). With \(q\ge1\), the product of this decay and the arc length \(2\pi R\) tends to \(0\) as \(R\to\infty\); therefore the limiting integral is zero. The residue theorem then translates this into the vanishing sum of residues used in the paper.

The structure of this paper is as follows: In Section 2, we present the main theorems and corollaries, while the proofs of the primary results are provided in subsequent sections. Section 3 details the Laurent expansions of the integrands of the considered contour integrals at their poles, preparing for the residue calculations. In Sections 4 and 5, we present the proofs of the two main theorems, respectively, along with illustrative corollaries and examples.

\section{Main Theorems and Corollaries}
In this section, we present a symmetry formula for cyclotomic multiple zeta values and a parity formula for cyclotomic multiple Hurwitz zeta values, from which a parity formula for cyclotomic multiple $t$-values can be derived. By further removing the product terms in the expressions, we also provide concise symmetry results.

To better describe the main results of this paper, we adopt the following concise notations: ($k_0:=q$)
\begin{align}
&B_m\left(\overrightarrow{\bfk}_{1,r};\overrightarrow{\bfm}_{1,r}\right):=(-1)^m\prod_{l=1}^r \binom{m_l+k_l-1}{k_l-1},\\
&C_{q,m,i,j}\left(\overrightarrow{\bfk}_{1,r};\overrightarrow{\bfm}_{1,r}\right):=(-1)^{|\overrightarrow{\bfk}_{0,j-1}|+|\overrightarrow{\bfm}_{j+1,r}|}\binom{q+k_j-m-i-1}{q-1}\prod_{l=1,\atop l\neq j}^r \binom{m_l+k_l-1}{k_l-1},\\
&D_{q,m,i,j}\left(\overrightarrow{\bfk}_{1,r};\overrightarrow{\bfm}_{1,r}\right):=(-1)^{|\overrightarrow{\bfk}_{0,j-1}|+|\overrightarrow{\bfm}_{j+1,r}|}\binom{q+k_j-m-i-2}{q-1}\prod_{l=1,\atop l\neq j}^r \binom{m_l+k_l-1}{k_l-1},\\
&\widehat{\Li}_{j+1}(x;a):=(-1)^j\Li_{j+1}(x^{-1};1-a)-x^{-1}\Li_{j+1}(x;a)\quad (a\in \mathbb{C}\setminus \Z),\\
&\widehat{\ti}_{j+1}(x):=(-1)^j\ti_{j+1}(x^{-1})-x^{-1}\ti_{j+1}(x).
\end{align}
The two main theorems of this paper are as follows (The symmetry may not be immediately evident from the formulas below, but it will become clearly visible in the later corollaries after removing all lower-order product terms):

\begin{thm}\label{mainthm-CMZVs} Let $k_0=q$ and $x_0=x$. For $(k_0,k_1,\ldots,k_r)\in \N^{r+1}$ and $x_0,x_1,\ldots,x_r$ are roots of unity with $(k_r,x_r)\neq(1,1)$ and $(k_0,x_0)\neq (1,1)$, we have
\begin{align}\label{equ-mainthm-CMZVs}
&\Li_{\overrightarrow{\bfk}_{0,r}}\left(\overrightarrow{\bfx}_{0,r}\right)+\sum_{j=0}^r (-1)^{|\overrightarrow{\bfk}_{0,j}|}\Li_{\overrightarrow{\bfk}_{j+1,r}}\left(\overrightarrow{\bfx}_{j+1,r}\right)\Li_{\overleftarrow{\bfk}_{0,j}}\left(\overleftarrow{\bfx}^{-1}_{0,j}\right)
\nonumber\\
&-\sum_{j+m=q,\atop j,m\geq 0} \sum_{|\bfm_{1,r}|=m,\atop \forall\ m_l\geq 0} B_m\left(\overrightarrow{\bfk}_{1,r};\overrightarrow{\bfm}_{1,r}\right) \left((-1)^j\Li_j\left(\widetilde{\bfx^{-1}_{0,r}}\right)+\Li_j\left(\widetilde{\bfx_{0,r}}\right)\right)\Li_{\overrightarrow{(\bfk+\bfm)}_{1,r}}\left(\overrightarrow{\bfx}_{1,r}\right)\nonumber\\
&-\sum_{j=1}^r\sum_{0\leq i+m\leq k_j,\atop i,m\geq 0} \sum_{|\overrightarrow{\bfm}_{1,r}|-m_j=m,\atop \forall\ m_l\geq 0} C_{q,m,i,j}\left(\overrightarrow{\bfk}_{1,r};\overrightarrow{\bfm}_{1,r}\right)\left((-1)^i\Li_i\left(\widetilde{\bfx^{-1}_{0,r}}\right)+\Li_i\left(\widetilde{\bfx_{0,r}}\right)\right)\nonumber\\
&\qquad\qquad\times \Li_{\overrightarrow{(\bfk+\bfm)}_{j+1,r}}\left(\overrightarrow{\bfx}_{j+1,r}\right)
\Li_{\overleftarrow{(\bfk+\bfm)}_{1,j-1},q+k_j-i-m}\left(\overleftarrow{\bfx}_{0,j-1}^{-1}\right)\nonumber\\
&=0,
\end{align}
where $\widetilde{\bfx_{0,r}}:=x_0x_1\cdots x_r,\ \widetilde{\bfx^{-1}_{0,r}}:=(x_0x_1\cdots x_r)^{-1}$, and $\Li_0(x)+\Li_0(x^{-1}):=-1$.
\end{thm}

\begin{thm}\label{mainthm-CMHZVs} Let $k_0:=q,\ a_0:=a$ and $x_0:=x$. For $(k_0,k_1,\ldots,k_r)\in \N^{r+1}$, $a_0,a_1,\ldots,a_r\in \mathbb{C}\setminus\Z$ and $x_0,x_1,\ldots,x_r$ are roots of unity with $(k_r,x_r)\neq(1,1)$ and $(k_0,x_0)\neq (1,1)$, we have
\begin{align}\label{equ-mainthm-CMHZVs}
&\widetilde{\bfx^{-1}_{0,r}}\Li_{\overrightarrow{\bfk}_{0,r}}\left(\overrightarrow{\bfx}_{0,r};\overrightarrow{\bfa}_{0,r}\right)+\sum_{j=0}^r (-1)^{|\overrightarrow{\bfk}_{0,j}|}\widetilde{\bfx_{j+1,r}^{-1}}\Li_{\overrightarrow{\bfk}_{j+1,r}}\left(\overrightarrow{\bfx}_{j+1,r};\overrightarrow{\bfa}_{j+1,r}\right)
\Li_{\overleftarrow{\bfk}_{0,j}}\left(\overleftarrow{\bfx}^{-1}_{0,j};\overleftarrow{(1-\bfa)}_{0,j}\right)\nonumber\\
&+\sum_{j+m=q-1,\atop j,m\geq 0}\sum_{|\overrightarrow{\bfm}_{1,r}|=m,\atop \forall\ m_l\geq 0} B_m\left(\overrightarrow{\bfk}_{1,r};\overrightarrow{\bfm}_{1,r}\right) \widehat{\Li}_{j+1}\left(\widetilde{\bfx_{0,r}};a\right)\Li_{\overrightarrow{(\bfk+\bfm)}_{1,r}}\left(\overrightarrow{\bfx}_{1,r};\overrightarrow{(1-a+\bfa)}_{1,r}\right)\nonumber\\
&+\sum_{j=1}^r \sum_{0\leq i+m\leq k_j-1,\atop i,m\geq 0}\sum_{|\overrightarrow{\bfm}_{1,r}|-m_j=m,\atop \forall\ m_l\geq 0}D_{q,m,i,j}\left(\overrightarrow{\bfk}_{1,r};\overrightarrow{\bfm}_{1,r}\right)\widehat{\Li}_{i+1}\left(\widetilde{\bfx_{0,r}};a_j\right)\nonumber\\
&\quad\times \Li_{\overrightarrow{(\bfk+\bfm)}_{j+1,r}}\left(\overrightarrow{\bfx}_{j+1,r};\overrightarrow{(1-a_j+\bfa)}_{j+1,r}\right)
\Li_{\overleftarrow{(\bfk+\bfm)}_{1,j-1},q+k_j-i-m-1}\left(\overleftarrow{\bfx}_{0,j-1}^{-1};\overleftarrow{(1+a_j-\bfa)}_{0,j-1}\right)\nonumber\\
&=0,
\end{align}
where $a_j-a_i\notin \N^-$ for $i< j\ (i,j\in \{0,1,2,\ldots,r\})$. Here $\widetilde{\bfx_{j,r}}:=x_jx_{j+1}\cdots x_r $ and $\widetilde{\bfx^{-1}_{j,r}}:=(x_jx_{j+1}\cdots x_r)^{-1}\ (0\leq j\leq r)$.
\end{thm}

The following examples can be derived from Theorems \ref{mainthm-CMZVs} and \ref{mainthm-CMHZVs}.

\begin{exa} Let \( x, y, z \) be roots of unity with \( x \neq 1 \). Then we have
\begin{align*}
	&\Li_{1,2,2}(x,y,z) - \Li_{2,2}(y,z) \Li_1(x^{-1}) - \Li_2(z) \Li_{2,1}(y^{-1},x^{-1}) - \Li_{2,2,1}(z^{-1},y^{-1},x^{-1}) \\
	&\quad - \Big( \Li_1(xyz) - \Li_1((xyz)^{-1}) \Big) \Li_{2,2}(y,z) - 2\Li_{3,2}(y,z) - 2\Li_{2,3}(y,z) \\
	&\quad - \Li_2(z) \Li_3(x^{-1}) + 2 \Li_3(z) \Li_2(x^{-1}) - 3 \Li_4(z) \Li_1(x^{-1}) \\
	&\quad + \Big( \Li_1(xyz) - \Li_1((xyz)^{-1}) \Big) \Big( \Li_2(z) \Li_2(x^{-1}) - 2 \Li_3(z) \Li_1(x^{-1}) \Big) \\
	&\quad + \Big( \Li_2(xyz) + \Li_2((xyz)^{-1}) \Big) \Li_2(z) \Li_1(x^{-1}) \\
	&\quad - \Li_{2,3}(y^{-1}, x^{-1}) - 2 \Li_{3,2}(y^{-1}, x^{-1}) - 3 \Li_{4,1}(y^{-1}, x^{-1}) \\
	&\quad + \Big( \Li_1(xyz) - \Li_1((xyz)^{-1}) \Big) \Big( \Li_{2,2}(y^{-1}, x^{-1}) + 2 \Li_{3,1}(y^{-1}, x^{-1}) \Big) \\
	&\quad + \Big( \Li_2(xyz) + \Li_2((xyz)^{-1}) \Big) \Li_{2,1}(y^{-1}, x^{-1}) = 0
\end{align*}
and
\begin{align*}
	&(xyz)^{-1} \Li_{1,2,2}(x,y,z; a,b,c) - (yz)^{-1} \Li_{2,2}(y,z; b,c) \Li_1(x^{-1}; 1-a) \\
	&\quad - z^{-1} \Li_2(z; c) \Li_{2,1}(y^{-1},x^{-1}; 1-b, 1-a) - \Li_{2,2,1}(z^{-1},y^{-1},x^{-1}; 1-c, 1-b, 1-a) \\
	&\quad + \widehat{\Li}_{1}(xyz; a) \Li_{2, 2}(y,z; 1-a+b, 1-a+c) \\
	&\quad - \widehat{\Li}_{1}(xyz; b) \Li_2(z; 1-b+c) \Li_2(x^{-1}; 1+b-a) \\
	&\quad + 2 \widehat{\Li}_{1}(xyz; b) \Li_3(z; 1-b+c) \Li_1(x^{-1}; 1+b-a) \\
	&\quad - \widehat{\Li}_{2}(xyz; b) \Li_{2}(z; 1-b+c) \Li_{1}(x^{-1}; 1+b-a) \\
	&\quad - \widehat{\Li}_{1}(xyz; c) \Li_{2,2}(y^{-1},x^{-1}; 1+c-b, 1+c-a) \\
	&\quad - 2 \widehat{\Li}_{1}(xyz; c) \Li_{3,1}(y^{-1},x^{-1}; 1+c-b, 1+c-a) \\
	&\quad - \widehat{\Li}_{2}(xyz; c) \Li_{2,1}(y^{-1},x^{-1}; 1+c-b, 1+c-a) = 0,
\end{align*}
where $a,b,c\in \mathbb{C}\setminus\Z$ and $b-a,c-a,c-b\notin \N^-$.
\end{exa}

From Theorem \ref{mainthm-CMZVs} it can be seen that the expression in \eqref{equ-mainthm-CMZVs} consists entirely of cyclotomic multiple zeta values. In Theorem \ref{mainthm-CMHZVs}, although \eqref{equ-mainthm-CMHZVs} also appears to involve only cyclotomic multiple Hurwitz zeta values at first glance, assigning specific values to all parameters $a_j$ can lead to the emergence of cyclotomic multiple zeta values alongside cyclotomic multiple Hurwitz zeta values even though the former may be regarded as a special case of the latter. For example, setting $a_0=a_1=\cdots=a_r=1/2$ in Theorem \ref{mainthm-CMHZVs} yields the symmetry result for cyclotomic multiple $t$-values:

\begin{cor}\label{maincor-CMtVs} Let $k_0=q$ and $x_0=x$. For $(k_0,k_1,\ldots,k_r)\in \N^{r+1}$ and $x_0,x_1,\ldots,x_r$ are roots of unity with $(k_r,x_r)\neq(1,1)$ and $(k_0,x_0)\neq (1,1)$, we have
\begin{align}\label{equ-mainthm-CMtVs}
&\widetilde{\bfx^{-1}_{0,r}}\ti_{\overrightarrow{\bfk}_{0,r}}\left(\overrightarrow{\bfx}_{0,r}\right)+\sum_{j=0}^r (-1)^{|\overrightarrow{\bfk}_{0,j}|}\widetilde{\bfx_{j+1,r}^{-1}}\ti_{\overrightarrow{\bfk}_{j+1,r}}\left(\overrightarrow{\bfx}_{j+1,r}\right)
\ti_{\overleftarrow{\bfk}_{0,j}}\left(\overleftarrow{\bfx}^{-1}_{0,j}\right)\nonumber\\
&+\sum_{j+m=q-1,\atop j,m\geq 0}\sum_{|\overrightarrow{\bfm}_{1,r}|=m,\atop \forall\ m_l\geq 0} B_m\left(\overrightarrow{\bfk}_{1,r};\overrightarrow{\bfm}_{1,r}\right) \widehat{\ti}_{j+1}\left(\widetilde{\bfx_{0,r}}\right)\Li_{\overrightarrow{(\bfk+\bfm)}_{1,r}}\left(\overrightarrow{\bfx}_{1,r}\right)\nonumber\\
&+\sum_{j=1}^r \sum_{0\leq i+m\leq k_j-1,\atop i,m\geq 0}\sum_{|\overrightarrow{\bfm}_{1,r}|-m_j=m,\atop \forall\ m_l\geq 0}D_{q,m,i,j}\left(\overrightarrow{\bfk}_{1,r};\overrightarrow{\bfm}_{1,r}\right)\widehat{\ti}_{i+1}\left(\widetilde{\bfx_{0,r}}\right)\nonumber\\
&\quad\times \Li_{\overrightarrow{(\bfk+\bfm)}_{j+1,r}}\left(\overrightarrow{\bfx}_{j+1,r}\right)
\Li_{\overleftarrow{(\bfk+\bfm)}_{1,j-1},q+k_j-i-m-1}\left(\overleftarrow{\bfx}_{0,j-1}^{-1}\right)\nonumber\\
&=0.
\end{align}
\end{cor}

It is evident from \eqref{equ-mainthm-CMtVs} that not only cyclotomic multiple \(t\)-values but also cyclotomic multiple zeta values appear in the expression. In email correspondence with Charlton, the author of this paper was informed that in Charlton and Hoffmann's work \cite{CharltonHoffman-MathZ2025}, where they established explicit formulas for the symmetry relations of regularized multiple \(t\)-values, classical multiple zeta values do indeed appear in the resulting expressions. However, by applying Murakami's profound theorem from \cite{Murakami2021}, which states that the motivic multiple $t$-value $t^\fm(k_1,\ldots,k_r)\ (\text{where all}\ k_j\in\{2,3\})$ generates all motivic multiple zeta values, one can ensure that their final formulas for the symmetry of regularized multiple \(t\)-values involve only regularized multiple \(t\)-values, that is no classical multiple zeta values appear. Yet Charlton also remarked that he found it surprising that such a seemingly natural symmetry result requires the machinery of motivic theory for its proof. Similarly, viewing this from Charlton's perspective, the author of this work may pose the question: \emph{Could there exist a particular family of cyclotomic multiple \(t\)-values that forms a basis for all cyclotomic multiple zeta values}?

It is clear that the maximum depth occurring in the cyclotomic multiple (Hurwitz) zeta/t-values in Theorems \ref{mainthm-CMZVs} and \ref{mainthm-CMHZVs}, as well as in Corollary \ref{maincor-CMtVs}, is \(r+1\), and there are only two such terms. All remaining terms are products of terms with depth \(\leq r\). Hence, from Theorems \ref{mainthm-CMZVs} and \ref{mainthm-CMHZVs}, along with Corollary \ref{maincor-CMtVs}, by removing all product terms in the expressions, we obtain the following concise symmetry results:
\begin{cor}\label{main-thm-corone}
Let $x_1,\ldots,x_r$ be roots of unity. For $\bfk_{1,r}=(k_1,\ldots,k_r)\in \N^r$ and $(k_r,x_r),\ (k_1,x_1)\neq (1,1)$, we have
\begin{align*}
&\Li_{k_1,\ldots,k_r}(x_1,\ldots,x_r)\equiv  (-1)^{k_1+\cdots+k_r-1}\Li_{k_r,\ldots,k_1}(x^{-1}_r,\ldots,x_1^{-1})\quad (\text{mod products}),
\end{align*}
where ``mod products" means discarding all product terms of cyclotomic multiple zeta values of smaller depth.
\end{cor}

\begin{cor}\label{main-thm-cortwo}
Let $x_1,\ldots,x_r$ be roots of unity. For $\bfk_{1,r}=(k_1,\ldots,k_r)\in \N^r$ and $(k_r,x_r),\ (k_1,x_1)\neq (1,1)$, we have
\begin{align*}
&\ti_{k_1,\ldots,k_r}(x_1,\ldots,x_r) \equiv  (-1)^{k_1+\cdots+k_r-1} (x_1\cdots x_r) \ti_{k_r,\ldots,k_1}(x^{-1}_r,\ldots,x_1^{-1})\quad (\text{mod products}),
\end{align*}
where ``mod products" means discarding all product terms of cyclotomic multiple zeta values and cyclotomic multiple $t$-values of smaller depth.
\end{cor}

\begin{cor}\label{main-thm-corthree}
Let $x_1,\ldots,x_r$ be roots of unity and $a_1,\ldots,a_r\in \mathbb{C}\setminus \Z$ satisfy $a_j-a_i\notin \N^-$ for $i< j\ (i,j\in \{1,2,\ldots,r\})$. For $\bfk_{1,r}=(k_1,\ldots,k_r)\in \N^r$ and $(k_r,x_r),\ (k_1,x_1)\neq (1,1)$, we have
\begin{align*}
&\Li_{k_1,\ldots,k_r}(x_1,\ldots,x_r;a_1,\ldots,a_r)\\
&\equiv  (-1)^{k_1+\cdots+k_r-1} (x_1\cdots x_r) \Li_{k_r,\ldots,k_1}(x^{-1}_r,\ldots,x_1^{-1};1-a_r,\ldots,1-a_1)\quad (\text{mod products}),
\end{align*}
where ``mod products" means discarding all product terms of cyclotomic multiple Hurwitz zeta values of smaller depth.
\end{cor}

\begin{re}\label{re-conj}
When $a_1=\cdots=a_r=a$ in Corollary \ref{main-thm-corthree}, this confirms the conjecture proposed by Rui in \cite{Rui2026} regarding the symmetry of cyclotomic multiple Hurwitz zeta values with $r$-variables and single-parameter (see \cite[Conjecture 4.6]{Rui2026}).

Clearly, Corollary \ref{main-thm-cortwo} can also be obtained by setting all $a_j = 1/2$ in Corollary \ref{main-thm-corthree}, whereas Corollary \ref{main-thm-corone} cannot be derived from Corollary \ref{main-thm-corthree}.
It should be emphasized that Charlton and Hoffman used generating function methods to derive the symmetry results for the regularization of multiple
$t$-values, see \cite[Thm. 1.1]{CharltonHoffman-MathZ2025}. Perhaps the approach used in the paper by Charlton and Hoffman can also be applied to study symmetry results for cyclotomic multiple Hurwitz zeta values. Interested readers are encouraged to explore this possibility.
It is worth emphasizing that by applying antipode relations, parity results for cyclotomic multiple zeta values, cyclotomic multiple $t$-values, and cyclotomic multiple Hurwitz zeta values can be further established. For recent progress on parity-related studies, refer to the literature \cite{Bouillot2014,LiXu2026,Hirose2025,IKZ2006,Panzer2017,Tsumura2004,Tsumura2008,Umezawa2025arxiv}.
\end{re}
It is important to emphasize that the term $\Li_{m+1}(x;1-a)-(-1)^m x\Li_{m+1}(x^{-1};a)$ in this work inherently possesses an equivalent formulation, as detailed in Theorem \ref{thm-main-two} below. (To ensure that the formulas presented in Theorems \ref{mainthm-CMZVs} and \ref{mainthm-CMHZVs} remain expressed entirely in terms of cyclotomic multiple Hurwitz zeta values, we deliberately avoid employing the identity given in \eqref{equ-thm-main-two}.)
\begin{thm}\label{thm-main-two} Let $x=e^{i\theta}$ and $\theta\in (0,2\pi)$. For $a\in \mathbb{C}\setminus \Z$ and $m\in \N_0$, we have
\begin{align}\label{equ-thm-main-two}
\Li_{m+1}(x;1-a)-(-1)^m x\Li_{m+1}(x^{-1};a)=\frac{\pi}{m!}\frac{d^m}{da^m}\left(ix^a-\cot(\pi a)x^a\right),
\end{align}
where $i^2=-1$.
\end{thm}
\begin{proof} From \cite[P. 123, Example 6.12]{WW2021}, we have
\begin{align*}
\sum_{n=-\infty}^\infty \frac{e^{in\theta}}{n-x}&=2\pi i\frac{e^{i\theta x}}{1-e^{2\pi ix}}=-\pi\frac{e^{ix(\theta-\pi)}}{\sin(\pi x)}\\
&=-\pi\cot(\pi a)x^a+\pi x^ai.
\end{align*}
On the other hand, noting the fact that
\begin{align*}
\sum_{n=-\infty}^\infty \frac{x^n}{n-a}&=\sum_{n=1}^\infty \frac{x^n}{n-a}+\sum_{n=-\infty}^0 \frac{x^n}{n-a}=\sum_{n=1}^\infty \frac{x^n}{n-a}-\sum_{n=1}^\infty \frac{x^{1-n}}{n+a-1}\\
&=\Li_1(x;1-a)-x\Li_1(x^{-1};a).
\end{align*}
Hence, we obtain
\begin{align*}
\Li_1(x;1-a)-x\Li_1(x^{-1};a)=-\pi\cot(\pi a)x^a+\pi x^ai.
\end{align*}
Finally, differentiating the above equation $m$ times with respect to a yields the desired result.
\end{proof}

It should be noted that Goncharov \cite[p.12]{Goncharov2001} introduced the following multi-series version of the series appearing in the proof of Theorem \ref{thm-main-two}:
\begin{align*}
B(\varphi_1,\ldots,\varphi_m|t_1,\ldots,t_m):=\sum_{-\infty<k_1<\cdots<k_m<\infty} \frac{e^{2\pi i(\varphi_1k_1+\cdots+\varphi_mk_m)}}{(k_1-t_1)\cdots (k_m-t_m)},
\end{align*}
where $\varphi_i\in (0,1)$. By investigating this function or similar generalizations, certain symmetry and parity results for multiple zeta values and their related variants can be derived, as seen in \cite{Hirose2025,CharltonHoffman-MathZ2025}.

\section{Lemmas and Expansions}
In this section, we provide certain Laurent expansions and Taylor expansions for the integrands of the contour integrals in \eqref{contour-integral-one} and \eqref{contour-integral-two}. These expansions will play a crucial role in the residue calculations in the following two sections.

Flajolet and Salvy \cite{Flajolet-Salvy} introduced a kernel function $\xi(s)$ satisfying two key properties:
(1) $\xi(s)$ is meromorphic on the whole complex plane.
(2) Over an infinite sequence of circles $|s| = \rho_k$ with $\rho_k \to \infty$, the growth condition $\xi(s) = o(s)$ holds.
Using these properties of the kernel function, Flajolet and Salvy established the following residue lemma.
\begin{lem}\emph{(cf.\ \cite{Flajolet-Salvy})}\label{lem-redisue-thm}
Let $\xi(s)$ be a kernel function and let $r(s)$ be a rational function which is $O(s^{-2})$ at infinity. Then
\begin{align}\label{residue-}
\sum\limits_{\alpha  \in O} {{\mathop{\rm Res}}{{\left( {r(s)\xi(s)},\alpha  \right)}}}  + \sum\limits_{\beta  \in S}  {{\mathop{\rm Res}}{{\left( {r(s)\xi(s)}, \beta  \right)}}}  = 0,
\end{align}
where $S$ is the set of poles of $r(s)$ and $O$ is the set of poles of $\xi(s)$ that are not poles $r(s)$. Here ${\mathop{\rm Re}\nolimits} s{\left( {r(s)},\alpha \right)} $ denotes the residue of $r(s)$ at $s= \alpha.$
\end{lem}
Notably, Lemma \ref{lem-redisue-thm} also holds under the weaker condition $r(s)\xi(s)=o(s^{-1})$. It is straightforward to show that the contour integrals corresponding to \eqref{contour-integral-one} and \eqref{contour-integral-two} must equal zero. Therefore, by the residue theorem, the sum of the residues of the integrand at all its poles in the complex plane must also equal zero. To compute these residues, it is necessary to determine the order of each pole. For this purpose, we need to clarify the Laurent expansion of the integrand at its poles. To better describe these expansions, we need to provide the following definitions: for $\overrightarrow{\bfk}_{1,r}=(k_1,\ldots,k_r)\in\N^r$, positive integer $n$, and $\bfx=(x_1,\ldots,x_r)\in \mathbb{C}^r$, we define the \emph{multiple harmonic sums} with $r$-variables by
\begin{align}
\zeta_n\left(\overrightarrow{\bfk}_{1,r};\overrightarrow{\bfx}_{1,r}\right)\equiv \zeta_n(k_1,\ldots,k_r;x_1,\ldots,x_r):=\sum\limits_{0<n_1<\cdots<n_r\leqslant n } \frac{x_1^{n_1}\cdots x_r^{n_r}}{n_1^{k_1}\cdots n_r^{k_r}}\label{MHSsr}.
\end{align}
If $n<r$ then $\zeta_n\left(\overrightarrow{\bfk}_{1,r};\overrightarrow{\bfx}_{1,r}\right):=0$ and ${\zeta _n}(\emptyset;\emptyset ):=1$. Furthermore, we can similarly define the Hurwitz-type generalization of the aforementioned multiple harmonic sums with $r$-variables as follows:
\begin{align}
\zeta_n\left(\overrightarrow{\bfk}_{1,r};\overrightarrow{\bfx}_{1,r};\overrightarrow{\bfa}_{1,r}\right)&\equiv \zeta_n(k_1,\ldots,k_r;x_1,\ldots,x_r;a_1,\ldots,a_r)\nonumber\\
&:=\sum\limits_{0<n_1<\cdots<n_r\leqslant n } \frac{x_1^{n_1}\cdots x_r^{n_r}}{(n_1+a_1)^{k_1}\cdots (n_r+a_r)^{k_r}}\label{MHtHSsr},
\end{align}
where $\overrightarrow{\bfa}_{1,r}=(a_1,\ldots,a_r)$ and $a_j\in \mathbb{C}\setminus \N^-\ (j=1,2,\ldots,r)$. Obviously, according to the definitions of \eqref{defn-CMHZV}, \eqref{defn-CMZV}, \eqref{MHSsr}, and \eqref{MHtHSsr}, when all $x_j$ are roots of unity and $(k_r, x_r) \neq (1, 1)$, letting $n$ approach infinity yields the limits
\begin{align}
\lim_{n\rightarrow \infty} \zeta_n\left(\overrightarrow{\bfk}_{1,r};\overrightarrow{\bfx}_{1,r}\right)=\Li_{\overrightarrow{\bfk}_{1,r}}\left(\overrightarrow{\bfx}_{1,r}\right)
\end{align}
and
\begin{align}
\lim_{n\rightarrow \infty} \zeta_n\left(\overrightarrow{\bfk}_{1,r};\overrightarrow{\bfx}_{1,r};\overrightarrow{\bfa}_{1,r}\right)=\Li_{\overrightarrow{\bfk}_{1,r}}\left(\overrightarrow{\bfx}_{1,r};\overrightarrow{(1+\bfa)}_{1,r}\right).
\end{align}

First, the following Laurent or Taylor expansions are given in \cite{LiXu2026,Rui-Xu2025}:

\begin{lem}(\cite[Eq. (2.6)]{Rui-Xu2025})\label{lem-rui-xu-one} For $n\in \Z$, if $|s-n|<1$, then
\begin{align}\label{LEPhi-function}
\Phi(s;x)=x^{-n} \left(\frac1{s-n}+\sum_{m=0}^\infty \Big((-1)^m\Li_{m+1}(x)-\Li_{m+1}\Big(x^{-1}\Big)\Big)(s-n)^m \right).
\end{align}
\end{lem}

\begin{lem}(\cite[Eq. (5.3)]{LiXu2026}) For $(k_1,\ldots,k_r)\in\N^r$ and roots of unity $x_1,\ldots,x_r$ with $(k_r,x_r)\neq(1,1)$, if $|s|<1$, we have
\begin{align}\label{Taylor-expansion-mhpolyf-cases}
&\Li_{k_1,\ldots,k_r}(x_1,\ldots,x_r;1+s)=\sum_{m=0}^\infty \sum_{|\overrightarrow{\bfn}_{1,r}|=m,\atop \forall n_l\geqslant 0} B_m\left(\overrightarrow{\bfk}_{1,r};\overrightarrow{\bfm}_{1,r}\right) \Li_{k_1+n_1,\ldots,k_r+n_r}(x_1,\ldots,x_r)s^m.
\end{align}
\end{lem}

\begin{lem}(\cite[Eq. (5.4)]{LiXu2026})\label{lem-Laurent-expansion-mhpolyf} Let $n\in \N$. For $(k_1,\ldots,k_r)\in\N^r$ and roots of unity $x_1,\ldots,x_r$ with $(k_r,x_r)\neq(1,1)$, if $|s+n|<1$, we have
\begin{align}\label{Laurent-expansion-mhpolyf}
&\Li_{k_1,\ldots,k_r}(x_1,\ldots,x_r;1+s)\nonumber\\
&=\sum_{j=0}^r \sum_{m=0}^\infty \sum_{|\overrightarrow{\bfn}_{1,r}|=m,\atop \forall n_l\geq 0}A_j\left(\overrightarrow{\bfk}_{1,r};\overrightarrow{\bfn}_{1,r}\right) (x_1\cdots x_r)^n\zeta_{n-1}\left(\overleftarrow{(\bfk+\bfn)}_{1,j};\overleftarrow{\bfx}_{1,j}^{-1}\right)\nonumber\\&\qquad\qquad\qquad\qquad\qquad\qquad\times\Li_{\overrightarrow{(\bfk+\bfn)}_{j+1,r}}\left(\overrightarrow{\bfx}_{j+1,r}\right)(s+n)^m
\nonumber\\
&\quad+\sum_{j=1}^r\sum_{m=0}^\infty \sum_{|\bfn_r|-n_j=m,\atop \forall n_l\geq 0} \widetilde{A}_j\left(\overrightarrow{\bfk}_{1,r};\overrightarrow{\bfn}_{1,r}\right) (x_1\cdots x_r)^n
\zeta_{n-1}\left(\overleftarrow{(\bfk+\bfn)}_{1,j-1};\overleftarrow{\bfx}_{1,j-1}^{-1}\right)\nonumber\\&\qquad\qquad\qquad\qquad\qquad\qquad\times
\Li_{\overrightarrow{(\bfk+\bfn)}_{j+1,r}}\left(\overrightarrow{\bfx}_{j+1,r}\right)(s+n)^{m-k_j},
\end{align}
where
\begin{align}
&A_j\left(\overrightarrow{\bfk}_{1,r};\overrightarrow{\bfn}_{1,r}\right):=(-1)^{|\overrightarrow{\bfk}_{1,j}|+|\overrightarrow{\bfn}_{j+1,r}|}\prod\limits_{l=1}^r \binom{n_l+k_l-1}{k_l-1} ,\\
&\widetilde{A}_j\left(\overrightarrow{\bfk}_{1,r};\overrightarrow{\bfn}_{1,r}\right):=(-1)^{|\overrightarrow{\bfk}_{1,j-1}|+|\overrightarrow{\bfn}_{j+1,r}|} \left(\prod\limits_{l=1,\atop l\neq j}^r \binom{n_l+k_l-1}{k_l-1} \right).
\end{align}
\end{lem}
In fact, the analytic part in \eqref{Laurent-expansion-mhpolyf} is redundant, as it will equal zero in the subsequent residue computation.
Through calculations similar to the proofs of \cite[Thms. 2.2 and 2.3]{LiXu2026}, we can obtain the following Taylor or Laurent expansions of the cyclotomic multiple Hurwitz zeta values:

\begin{lem}\label{lem-Talyor-expansion-CMHZ-single-a} Let $a,a_1,\ldots,a_r\in \mathbb{C}\setminus \N^-$ and $a_j-a\notin \N^-$ for any $j=1,2,\ldots,r$. For $(k_1,\ldots,k_r)\in\N^r$ and roots of unity $x_1,\ldots,x_r$ with $(k_r,x_r)\neq(1,1)$, if $|s+a|<1$, we have
\begin{align}\label{Talyor-expansion-CMHZ-single-a}
&\Li_{k_1,\ldots,k_r}(x_1,\ldots,x_r;1+a_1+s,\ldots,1+a_r+s)\nonumber\\
&=\sum_{m=0}^\infty \sum_{|\overrightarrow{\bfm}_{1,r}|=m,\atop \forall\ m_l\geq 0} B_m\left(\overrightarrow{\bfk}_{1,r};\overrightarrow{\bfm}_{1,r}\right)\Li_{\overrightarrow{(\bfk+\bfm)}_{1,r}}\left(\overrightarrow{\bfx}_{1,r};\overrightarrow{(1+\bfa-a)}_{1,r}\right)(s+a)^m.
\end{align}
\end{lem}
\begin{proof}
By directly computing the transformation of the denominator $n_j + s + a_j$ in the definition \eqref{defn-CMHZV} of the cyclotomic multiple Hurwitz zeta function into $n_j + a_j - a + s + a$, we can deduce that
\begin{align*}
&\Li_{k_1,\ldots,k_r}(x_1,\ldots,x_r;1+a_1+s,\ldots,1+a_r+s)\\
&=\sum_{0<n_1<\cdots<n_r} \frac{x_1^{n_1}\cdots x_r^{n_r}}{(n_1+a_1-a+s+a)^{k_1}\cdots (n_r+a_r-a+s+a)^{k_r}}\\
&=\sum_{0<n_1<\cdots<n_r} \frac{x_1^{n_1}\cdots x_r^{n_r}}{(n_1+a_1-a)^{k_1}\cdots (n_r+a_r-a)^{k_r}}\frac{1}{\left(1+\frac{s+a}{n_1+a_1-a}\right)^{k_1}\cdots \left(1+\frac{s+a}{n_r+a_r-a}\right)^{k_r}}\\
&=\sum_{0<n_1<\cdots<n_r} \frac{x_1^{n_1}\cdots x_r^{n_r}}{(n_1+a_1-a)^{k_1}\cdots (n_r+a_r-a)^{k_r}}\prod_{l=1}^r \left(\sum_{m_l=0}^\infty \binom{m_l+k_l-1}{k_l-1} \left(-\frac{s+a}{n_l+a_l-a}\right)^{m_l}\right)\\
&=\sum_{m_1,\ldots,m_r=0}^\infty \left(\prod_{l=1}^r (-1)^{m_l}\binom{m_l+k_l-1}{k_l-1}\right) \sum_{0<n_1<\cdots<n_r} \frac{x_1^{n_1}\cdots x_r^{n_r}(s+a)^{m_1+\cdots+m_r}}{(n_1+a_1-a)^{k_1+m_1}\cdots (n_r+a_r-a)^{k_r+m_r}}\\
&=\sum_{m=0}^\infty \sum_{|\overrightarrow{\bfm}_{1,r}|=m,\atop \forall\ m_l\geq 0}  B_m\left(\overrightarrow{\bfk}_{1,r};\overrightarrow{\bfm}_{1,r}\right)\Li_{\overrightarrow{(\bfk+\bfm)}_{1,r}}\left(\overrightarrow{\bfx}_{1,r};\overrightarrow{(1+\bfa-a)}_{1,r}\right)(s+a)^m.
\end{align*}
This completes the proof of Lemma \ref{lem-Talyor-expansion-CMHZ-single-a}.
\end{proof}

\begin{lem}\label{lem-Talyor-expansion-CMHZSs} Let $n\in \N$. For $(k_1,\ldots,k_r)\in\N^r$ and roots of unity $x_1,\ldots,x_r$ with $(k_r,x_r)\neq(1,1)$, if $|s+n|<1$, we have
\begin{align}\label{Talyor-expansion-mhpolyf}
&\Li_{k_1,\ldots,k_r}(x_1,\ldots,x_r;1+a_1+s,\ldots,1+a_r+s)\nonumber\\
&=\sum_{j=0}^r\sum_{m=0}^\infty \sum_{|\overrightarrow{\bfm}_{1,r}|=m,\atop \forall\ m_l\geq 0}  (-1)^{|\overrightarrow{\bfk}_{1,j}|+|\overrightarrow{\bfm}_{1,j}|}B_m\left(\overrightarrow{\bfk}_{1,r};\overrightarrow{\bfm}_{1,r}\right)(x_1\cdots x_r)^n \frac{\Li_{\overrightarrow{(\bfk+\bfm)}_{j+1,r}}\left(\overrightarrow{\bfx}_{j+1,r};\overrightarrow{\bfa}_{j+1,r}\right)}{x_{j+1}\cdots x_r}\nonumber\\
&\qquad\qquad\qquad\qquad\qquad\times \zeta_{n-1}\left(\overleftarrow{(\bfk+\bfm)}_{1,j};\overleftarrow{\bfx}_{1,j}^{-1};-\overleftarrow{\bfa}_{1,j}\right)(s+n)^m,
\end{align}
where $a_1,\ldots,a_r\in \mathbb{C}\setminus \Z$ and $-\overleftarrow{\bfa}_{1,j}:=(-a_j,-a_{j-1},\ldots,-a_1)$.
\end{lem}
\begin{proof}
A straightforward calculation yields
\begin{align*}
&\Li_{k_1,\ldots,k_r}(x_1,\ldots,x_r;1+a_1+s,\ldots,1+a_r+s)\\
&=\sum_{0<n_1<\cdots<n_r} \frac{x_1^{n_1}\cdots x_r^{n_r}}{(n_1+a_1-n+s+n)^{k_1}\cdots (n_r+a_r-n+s+n)^{k_r}}\\
&=\sum_{0<n_1<\cdots<n_r} \frac{x_1^{n_1}\cdots x_r^{n_r}}{(n_1+a_1-n)^{k_1}\cdots (n_r+a_r-n)^{k_r}}\frac{1}{\left(1+\frac{s+n}{n_1+a_1-n}\right)^{k_1}\cdots \left(1+\frac{s+n}{n_r+a_r-n}\right)^{k_r}}\\
&=\sum_{m_1,\ldots,m_r=0}^\infty \left(\prod_{l=1}^r (-1)^{m_l}\binom{m_l+k_l-1}{k_l-1}\right)\sum_{0<n_1<\cdots<n_r} \frac{x_1^{n_1}\cdots x_r^{n_r}(s+n)^{m_1+\cdots+m_r}}{(n_1+a_1-n)^{k_1+m_1}\cdots (n_r+a_r-n)^{k_r+m_r}}\\
&=\sum_{m=0}^\infty \sum_{|\overrightarrow{\bfm}_{1,r}|=m,\atop \forall\ m_l\geq 0}  B_m\left(\overrightarrow{\bfk}_{1,r};\overrightarrow{\bfm}_{1,r}\right)\sum_{-(n-1)<n_1<\cdots<n_r} \frac{x_1^{n_1+n-1}\cdots x_r^{n_r+n-1}(s+n)^m}{(n_1+a_1-1)^{k_1+m_1}\cdots (n_r+a_r-1)^{k_r+m_r}}.
\end{align*}
On the other hand, by truncating the rightmost summation $-(n-1)<n_1<\cdots<n_r$ at zero, we obtain
\begin{align*}
&\sum_{-(n-1)<n_1<\cdots<n_r} \frac{x_1^{n_1}\cdots x_r^{n_r}}{(n_1+a_1-1)^{k_1}\cdots (n_r+a_r-1)^{k_r}}\\
&=\sum_{j=0}^r \sum_{-(n-1)<n_1<\cdots<n_j\leq 0<n_{j+1}<\cdots<n_r} \frac{x_1^{n_1}\cdots x_r^{n_r}}{(n_1+a_1-1)^{k_1}\cdots (n_r+a_r-1)^{k_r}}\\
&=\sum_{j=0}^r \sum_{-(n-1)<n_1<\cdots<n_j\leq 0} \frac{x_1^{n_1}\cdots x_j^{n_j}}{(n_1+a_1-1)^{k_1}\cdots (n_j+a_j-1)^{k_j}}\\&\qquad\qquad\times\sum_{0<n_{j+1}<\cdots<n_r} \frac{x_{j+1}^{n_{j+1}}\cdots x_r^{n_r}}{(n_{j+1}+a_{j+1}-1)^{k_{j+1}}\cdots (n_r+a_r-1)^{k_r}}\\
&=\sum_{j=0}^r \sum_{n-1>n_1>\cdots>n_j\geq 0} \frac{x_1^{-n_1}\cdots x_j^{-n_j}}{(-n_1+a_1-1)^{k_1}\cdots (-n_j+a_j-1)^{k_j}}\Li_{\overrightarrow{\bfk}_{j+1,r}}(\overrightarrow{\bfx}_{j+1,r};\overrightarrow{\bfa}_{j+1,r})\\
&=\sum_{j=0}^r (-1)^{|\overrightarrow{\bfk}_{1,j}|}\sum_{0<n_j<\cdots<n_1<n} \frac{x_1^{1-n_1}\cdots x_j^{1-n_j}}{(n_1-a_1)^{k_1}\cdots (n_j-a_j)^{k_j}}\Li_{\overrightarrow{\bfk}_{j+1,r}}(\overrightarrow{\bfx}_{j+1,r};\overrightarrow{\bfa}_{j+1,r})\\
&=\sum_{j=0}^r (-1)^{|\overrightarrow{\bfk}_{1,j}|} (x_1\cdots x_j) \zeta_{n-1}\left(\overleftarrow{\bfk}_{1,j};\overleftarrow{\bfx}^{-1}_{1,j};-\overleftarrow{\bfa}_{1,j}\right)\Li_{\overrightarrow{\bfk}_{j+1,r}}(\overrightarrow{\bfx}_{j+1,r};\overrightarrow{\bfa}_{j+1,r}).
\end{align*}
Therefore, combining the two equations above yields the desired result.
\end{proof}

\begin{lem}\label{lem-Laurent-expansion-CMHZSs} Let $n\in\N$ and $a_1,\ldots,a_r\in \mathbb{C}\setminus \Z$ satisfy $a_j-a_i\notin \N^-$ for any $i< j\ (i,j\in \{1,2,\ldots,r\})$. For $(k_1,\ldots,k_r)\in\N^r$ and roots of unity $x_1,\ldots,x_r$ with $(k_r,x_r)\neq(1,1)$, if $|s+n+a_j|<1$, we have
\begin{align}\label{Laurent-expansion-Cmhpolyf}
&\Li_{k_1,\ldots,k_r}(x_1,\ldots,x_r;1+a_1+s,\ldots,1+a_r+s)\nonumber\\
&=\sum_{m=0}^{k_{j}-1} \sum_{|\bfn_r|-n_j=m,\atop \forall n_l\geq 0} \widetilde{A}_j\left(\overrightarrow{\bfk}_{1,r};\overrightarrow{\bfn}_{1,r}\right) (x_1\cdots x_r)^n
\zeta_{n-1}\left(\overleftarrow{(\bfk+\bfn)}_{1,j-1};\overleftarrow{\bfx}_{1,j-1}^{-1};\overleftarrow{(a_j-\bfa)}_{1,j-1}\right)\nonumber\\&\qquad\qquad\qquad\qquad\times
\Li_{\overrightarrow{(\bfk+\bfn)}_{j+1,r}}\left(\overrightarrow{\bfx}_{j+1,r};\overrightarrow{(1-a_j+\bfa)}_{j+1,r}\right)(s+n+a_j)^{m-k_j}\nonumber\\&\quad+\cdots,
\end{align}
where  ``$\cdots$" indicates the omission of constant terms and higher-order infinitesimals of $(s+n+a_j)$ as they are not needed in the subsequent residue calculation.
\end{lem}
\begin{proof}
The proof of this lemma proceeds similarly to the one above. It involves truncating the summation $0 < n_1 < \dots < n_r$ at $n$, and then expanding the denominator $(n_i + a_i + s)$ as $(n_i + a_i - a_j - n + s + n + a_j)\ (i\neq j)$ into a power series. Through a somewhat lengthy calculation, the desired result is obtained. The details are left to the reader.
\end{proof}

Finally, to conclude this section, we provide the expansion of the $\Phi(s;x)$ function around an arbitrary non-integer complex point, which will also be utilized in our subsequent residue calculations.
\begin{lem}\label{lem-xu-yu-one}
Let $x$ be a root of unity. For $n\in\N_0$ and $a\in \mathbb{C}\setminus \Z$, if $|s+n+a|<1$, then
		\begin{align}\label{equ-extend-xuzhou-two}
			\Phi(s;x)=x^n\sum_{m=0}^\infty \left((-1)^m \Li_{m+1}(x;1-a)-x \Li_{m+1}\Big(x^{-1};a\Big) \right)(s+n+a)^m.
		\end{align}
\end{lem}
\begin{proof}
The proof of this lemma also follows from a straightforward calculation together with the definition of $\Phi(s;x)$, as demonstrated below.
\begin{align*}
\Phi(s;x)&=\sum_{k=0}^\infty \left\{\frac{x^k}{k-n-a+s+n+a}-\frac{x^{-k}}{k+n+a-s-n-a}\right\}-\frac1{s+n+a-n-a}\\
&=\sum_{k=0}^\infty \left\{\frac{x^k}{k-n-a}\frac1{1+\frac{s+n+a}{k-n-a}}-\frac{x^{-k}}{k+n+a}\frac1{1-\frac{s+n+a}{k+n+a}}\right\}+\frac{1}{n+a}\frac{1}{1-\frac{s+n+a}{n+a}}\\
&=\sum_{k=0}^\infty \left\{\frac{x^k}{k-n-a}\sum_{m=0}^\infty \left(-\frac{s+n+a}{k-n-a}\right)^m-\frac{x^{-k}}{k+n+a}\sum_{m=0}^\infty \left(\frac{s+n+a}{k+n+a}\right)^m \right\}\\
&\quad+\frac1{n+a}\sum_{m=0}^\infty \left(\frac{s+n+a}{n+a}\right)^m\\
&=\sum_{m=0}^\infty (s+n+a)^m \left\{(-1)^m\sum_{k=0}^\infty \frac{x^k}{(k-n-a)^{m+1}}-\sum_{k=0}^\infty \frac{x^{-k}}{(k+n+a)^{m+1}} +\frac{1}{(n+a)^{m+1}}\right\}\\
&=\sum_{m=0}^\infty (s+n+a)^m \left\{ (-1)^m\sum_{k=0}^n \frac{x^k}{(k-n-a)^{m+1}}+(-1)^m\sum_{k=n+1}^\infty \frac{x^k}{(k-n-a)^{m+1}}\atop -\sum_{k=1}^\infty \frac{x^{-k}}{(k+n+a)^{m+1}} \right\}\\
&=\sum_{m=0}^\infty (s+n+a)^m \left\{(-1)^m \sum_{j=0}^n \frac{x^{n-j}}{(-j-a)^{m+1}}+(-1)^m\sum_{j=1}^\infty \frac{x^{j+n}}{(j-a)^{m+1}}\atop -\sum_{j=n+1}^\infty \frac{x^{n-j}}{(j+a)^{m+1}} \right\}\\
&=\sum_{m=0}^\infty (s+n+a)^m \left\{(-1)^m \sum_{j=1}^\infty \frac{x^{j+n}}{(j-a)^{m+1}}-\sum_{j=0}^\infty \frac{x^{n-j}}{(j+a)^{m+1}}\right\}.
\end{align*}
Building upon the definition of the cyclotomic single Hurwitz zeta values, the desired result can then be obtained.
\end{proof}

In Sections \ref{sec-proof-one} and \ref{sec-proof-two}, we will apply the expansions obtained above to calculate the residues of the contour integrals \eqref{contour-integral-one} and \eqref{contour-integral-two}, thereby providing detailed proofs of Theorems \ref{mainthm-CMZVs} and \ref{mainthm-CMHZVs}, respectively.

\section{Symmetry Results of Cyclotomic Multiple Zeta Values}\label{sec-proof-one}

In this section, we will prove Theorem \ref{mainthm-CMZVs} by considering the contour integral \eqref{contour-integral-one}. Let
\begin{align*}
F_q(s):=\frac{\Phi(s;x)\Li_{k_1,\ldots,k_r}(x_1,\ldots,x_r;s+1)}{s^q}.
\end{align*}
{\bf Proof of Theorem \ref{mainthm-CMZVs}.}
Clearly, the only singularities of the function $F_q(s)$ are poles at the integers. At a positive integer $n\in \N$, the pole $s=n$ is simple and by the expansion \eqref{LEPhi-function}, the residue is
\begin{align}\label{equ-proof-thm-one-residu-1}
{\rm Res}[F_q(s),n]&=\lim_{s\rightarrow n} (s-n) \frac{\Phi(s;x)\Li_{k_1,\ldots,k_r}(x_1,\ldots,x_r;s+1)}{s^q} \nonumber\\
&=\frac{x^{-n}}{n^q}\sum_{0<n_1<\cdots<n_r} \frac{x_1^{n_1}\cdots x_r^{n_r}}{(n+n_1)^{k_1}\cdots (n+n_r)^{k_r}}\nonumber\\
&=\frac{(xx_1\cdots x_r)^{-n}}{n^q} \sum_{n<n_1<\cdots<n_r} \frac{x_1^{n_1}\cdots x_r^{n_r}}{n_1^{k_1}\cdots n_r^{k_r}}.
\end{align}
Clearly, summing both sides of \eqref{equ-proof-thm-one-residu-1} over $n$ from 1 to infinity yields
\begin{align*}
\sum_{n=1}^\infty {\rm Res}[F_q(s),n]=\Li_{\overrightarrow{\bfk}_{0,r}}\left(\widetilde{\bfx^{-1}_{0,r}},\overrightarrow{\bfx}_{1,r}\right)\quad (k_0=q,\ x_0=x).
\end{align*}
The pole $s=0$ has order $q+1$. Using \eqref{LEPhi-function} and \eqref{Taylor-expansion-mhpolyf-cases}, an elementary calculations shows that the residue is
\begin{align}
&{\rm Res}[F_q(s),0]=\frac1{q!}\lim_{s\rightarrow 0} \frac{d^q}{ds^q}\left\{s^{q+1}\frac{\Phi(s;x)\Li_{k_1,\ldots,k_r}(x_1,\ldots,x_r;s+1)}{s^q}\right\}\nonumber\\
&=-\sum_{j+m=q,\atop j,m\geq 0} \sum_{|\bfm_{1,r}|=m,\atop \forall\ m_l\geq 0} B_m\left(\overrightarrow{\bfk}_{1,r};\overrightarrow{\bfm}_{1,r}\right) \left((-1)^j\Li_j\left(\widetilde{\bfx^{-1}_{0,r}}\right)+\Li_j\left(\widetilde{\bfx_{0,r}}\right)\right)\Li_{\overrightarrow{(\bfk+\bfm)}_{1,r}}\left(\overrightarrow{\bfx}_{1,r}\right).
\end{align}
For a positive integer $n$, the pole at $s=-n$ has order $K+1$, where ($K:=\max\{k_1,k_2,\ldots,k_r\}$). Using \eqref{LEPhi-function} and \eqref{Laurent-expansion-mhpolyf}, and performing a rather lengthy calculation, we find that the residue is
\begin{align}\label{equ-proof-thm-1-residue-3}
&{\rm Res}[F_q(s),-n]=\frac1{K!}\lim_{s\rightarrow -n} \frac{d^K}{ds^K}\left\{(s+n)^{K+1}\frac{\Phi(s;x)\Li_{k_1,\ldots,k_r}(x_1,\ldots,x_r;s+1)}{s^q}\right\}\nonumber\\
&=\sum_{j=0}^r (-1)^{|\overrightarrow{\bfk}_{0,j}|} \Li_{\overrightarrow{\bfk}_{j+1,r}}\left(\overrightarrow{\bfx}_{j+1,r}\right) (\widetilde{\bfx_{0,r}})^n\frac{\zeta_{n-1}\left(\overleftarrow{\bfk}_{1,j};\overleftarrow{\bfx}^{-1}_{1,j}\right)}{n^q}\nonumber\\
&\quad+\sum_{j=1}^r\sum_{m=0}^{k_j}\sum_{|\overrightarrow{\bfm}_{1,r}|-m_j=m,\atop \forall\ m_l\geq 0} C_{q,m,i,j}\left(\overrightarrow{\bfk}_{1,r};\overrightarrow{\bfm}_{1,r}\right) \Li_{\overrightarrow{(\bfk+\bfn)}_{j+1,r}}\left(\overrightarrow{\bfx}_{j+1,r}\right)\nonumber\\&\qquad\qquad\qquad\qquad\qquad\qquad\times(\widetilde{\bfx_{0,r}})^n
\frac{\zeta_{n-1}\left(\overleftarrow{(\bfk+\bfn)}_{1,j-1};\overleftarrow{\bfx}^{-1}_{1,j-1}\right)}{n^{q+k_j-m}} \nonumber\\
&\quad+\sum_{j=1}^r\sum_{0\leq i+m\leq k_j-1,\atop i,m\geq 0} \sum_{|\overrightarrow{\bfm}_{1,r}|-m_j=m,\atop \forall\ m_l\geq 0} D_{q,m,i,j}\left(\overrightarrow{\bfk}_{1,r};\overrightarrow{\bfm}_{1,r}\right)\left((-1)^i\Li_{i+1}\left(\widetilde{\bfx^{-1}_{0,r}}\right)-\Li_{i+1}\left(\widetilde{\bfx_{0,r}}\right)\right)\nonumber\\
&\qquad\qquad\qquad\qquad\times \Li_{\overrightarrow{(\bfk+\bfn)}_{j+1,r}}\left(\overrightarrow{\bfx}_{j+1,r}\right)(\widetilde{\bfx_{0,r}})^n
\frac{\zeta_{n-1}\left(\overleftarrow{(\bfk+\bfn)}_{1,j-1};\overleftarrow{\bfx}^{-1}_{1,j-1}\right)}{n^{q+k_j-m-1-i}},
\end{align}
where in the process of computing the residue, we have used the following formula:
\begin{align}
\lim_{s\rightarrow -n} \frac{d^k}{d s^k} \frac{(s+n)^m}{s^q}=(-1)^qk!\binom{q+k-m-1}{q-1} \frac{1}{n^{q+k-m}}\quad(m\leq k).
\end{align}
It is not difficult to see that, provided that we stipulate that $\Li_0(x) + \Li_0(x^{-1}) = -1$, the expressions in the second and third terms of \eqref{equ-proof-thm-1-residue-3} can be combined into one.

By Lemma \ref{lem-redisue-thm}, we know that
\[\sum_{n=1}^\infty \left({\rm Res}[F_q(s),n]+{\rm Res}[F_q(s),-n]\right)+{\rm Res}[F_q(s),0]=0.\]
Noting the fact that
\begin{align*}
C_{q,m,i,j}\left(\overrightarrow{\bfk}_{1,r};\overrightarrow{\bfm}_{1,r}\right)=D_{q,m,i-1,j}\left(\overrightarrow{\bfk}_{1,r};\overrightarrow{\bfm}_{1,r}\right),
\end{align*}
then combining these three contributions yields the statement of Theorem \ref{mainthm-CMZVs} after replacing $x$ by $\widetilde{\bfx_{0,r}^{-1}}$. \hfill$\square$

Setting $r=1, k_1=k$ and $x_1=y$ in Theorem \ref{mainthm-CMZVs} yields the following corollary, namely the symmetry formula for cyclotomic double zeta values.
\begin{cor}\label{cor-proof-SRCMZV}
Let $x,y$ be roots of unity. For $q,k\in\N$ and $(q,x), (k,y)\neq (1,1)$, we have
\begin{align*}
&\Li_{q,k}(x,y)+(-1)^q\Li_k(y)\Li_q\left(\frac1{x}\right)+(-1)^{q+k}\Li_{k,q}\left(\frac1{y},\frac1{x}\right)\nonumber\\
&-(-1)^q \sum_{l=0}^k \binom{q+k-l-1}{q-1}\left( (-1)^l\Li_l\left(\frac1{xy}\right)+\Li_{l}(xy)\right)\Li_{q+k-l}\left(\frac1{x}\right)\nonumber\\
&-\sum_{l+m=q,\atop l,m\geq 0} (-1)^m \binom{m+k-1}{k-1}\left( (-1)^l\Li_l\left(\frac1{xy}\right)+\Li_{l}(xy)\right)\Li_{k+m}(y)=0,
\end{align*}
where $\Li_0(x)+\Li_0(x^{-1})=-1$.
\end{cor}
By further applying the stuffle relation
\[
\Li_q(x)\Li_k(y) = \Li_{q,k}(x,y) + \Li_{k,q}(y,x) + \Li_{k+q}(xy)
\]
to Corollary \ref{cor-proof-SRCMZV}, we obtain the following parity formula for cyclotomic double Hurwitz zeta values, which is equivalent to Theorem 3.1 in \cite{Rui-Xu2025}.
\begin{cor}\label{cor-proof-PFCMZV}
Let $x,y$ be roots of unity. For $q,k\in\N$ and $(q,x), (k,y)\neq (1,1)$, we have
\begin{align*}
&\Li_{k,q}(y,x)-(-1)^{q+k}\Li_{k,q}\left(\frac1{y},\frac1{x}\right)\nonumber\\
&=\Li_q(x)\Li_k(y)+(-1)^q \Li_k(y)\Li_q\left(\frac1{x}\right)-\Li_{k+q}(xy)\nonumber\\
&\quad-(-1)^q \sum_{l=0}^k \binom{q+k-l-1}{q-1}\left( (-1)^l\Li_l\left(\frac1{xy}\right)+\Li_{l}(xy)\right)\Li_{q+k-l}\left(\frac1{x}\right)\nonumber\\
&\quad-(-1)^q\sum_{l=0}^q \binom{q+k-l-1}{k-1}\left( (-1)^l\Li_l\left(\frac1{xy}\right)+\Li_{l}(xy)\right)\Li_{q+k-l}(y),
\end{align*}
where $\Li_0(x)+\Li_0(x^{-1})=-1$.
\end{cor}

\section{Symmetry Results of Cyclotomic Multiple Hurwitz Zeta Values}\label{sec-proof-two}

In this section, we will prove Theorem \ref{mainthm-CMHZVs} by considering the residue computations of the contour integral \eqref{contour-integral-two}. For convenience, we let
\begin{align*}
G_q(s):=\frac{\Phi(s;x)\Li_{k_1,\ldots,k_r}(x_1,\ldots,x_r;1+a_1+s,\ldots,1+a_r+s)}{(s+a)^q}.
\end{align*}
{\bf Proof of Theorem \ref{mainthm-CMHZVs}.} From the definition, it is not difficult to see that the integrand $G_q(s)$ has, in addition to simple poles at all integer points on the complex plane, a pole of order $q$ at $-a$, and poles of order $k_j$ at $-n - a_j$ (where $n$ is any positive integer and $j = 1,2,\ldots,r$).

Applying \eqref{LEPhi-function} and \eqref{Talyor-expansion-mhpolyf} and performing elementary calculations, we can determine the residues of the function at the simple poles $n\in \N_0$ and $-n\in \N^-$, which are given respectively as follows. For $n\in\N_0$,
\begin{align}\label{equ-proof-thm-two-residu-1}
{\rm Res}[G_q(s),n]&=\lim_{s\rightarrow n} (s-n) \frac{\Phi(s;x)\Li_{k_1,\ldots,k_r}(x_1,\ldots,x_r;1+a_1+s,\ldots,1+a_r+s)}{(s+a)^q} \nonumber\\
&=\frac{(xx_1\cdots x_r)^{-n}}{(n+a)^q} \sum_{n<n_1<\cdots<n_r} \frac{x_1^{n_1}\cdots x_r^{n_r}}{(n_1+a_1)^{k_1}\cdots (n_r+a_r)^{k_r}}
\end{align}
and for $n\in\N$,
\begin{align}\label{equ-proof-thm-two-residu-2}
{\rm Res}[G_q(s),-n]&=\lim_{s\rightarrow -n} (s+n) \frac{\Phi(s;x)\Li_{k_1,\ldots,k_r}(x_1,\ldots,x_r;1+a_1+s,\ldots,1+a_r+s)}{(s+a)^q} \nonumber\\
&=\sum_{j=0}^r (-1)^{|\overrightarrow{\bfk}_{0,j}|}\frac{\Li_{\overrightarrow{\bfk}_{j+1,r}}\left(\overrightarrow{\bfx}_{j+1,r};\overrightarrow{\bfa}_{j+1,r}\right)}{x_{j+1}\cdots x_r}\nonumber\\&\qquad\qquad\qquad\times\frac{\zeta_{n-1}\left(\overleftarrow{\bfk}_{1,j};\overleftarrow{\bfx}_{1,j}^{-1};-\overleftarrow{\bfa}_{1,j}\right)}{(n-a)^q}(\widetilde{\bfx_{0,r}})^n.
\end{align}
In particular, summing \eqref{equ-proof-thm-two-residu-1} over $n$ from 0 to infinity yields
\begin{align*}
\sum_{n=0}^\infty {\rm Res}[G_q(s),n]&=\sum_{n=0}^\infty\frac{(xx_1\cdots x_r)^{-n}}{(n+a)^q} \sum_{n<n_1<\cdots<n_r} \frac{x_1^{n_1}\cdots x_r^{n_r}}{(n_1+a_1)^{k_1}\cdots (n_r+a_r)^{k_r}}\\
&=\sum_{n=1}^\infty \frac{(xx_1\cdots x_r)^{1-n}}{(n+a-1)^q} \sum_{n-1<n_1<\cdots<n_r} \frac{x_1^{n_1}\cdots x_r^{n_r}}{(n_1+a_1)^{k_1}\cdots (n_r+a_r)^{k_r}}\\
&=\sum_{n=1}^\infty \frac{(xx_1\cdots x_r)^{1-n}}{(n+a-1)^q} \sum_{n<n_1<\cdots<n_r} \frac{x_1^{n_1-1}\cdots x_r^{n_r-1}}{(n_1-1+a_1)^{k_1}\cdots (n_r-1+a_r)^{k_r}}\\
&=x \sum_{0<n<n_1<\cdots<n_r} \frac{(xx_1\cdots x_r)^{-n}x_1^{n_1}\cdots x_r^{n_r}}{(n+a-1)^q(n_1+a-1)^{k_1}\cdots (n_r+a_r-1)^{k_r}}\quad (q:=k_0)\\
&=x\Li_{\overrightarrow{\bfk}_{0,r}}\left(\widetilde{\bfx^{-1}_{0,r}},\overrightarrow{\bfx}_{1,r};\overrightarrow{\bfa}_{0,r}\right)\quad (a:=a_0).
\end{align*}

Using \eqref{equ-extend-xuzhou-two} and \eqref{Talyor-expansion-CMHZ-single-a}, the residue at the pole of order $q$ located at $-a$ can be computed as
\begin{align}\label{equ-proof-thm-two-residu-3}
&{\rm Res}[G_q(s),-a]=\frac1{(q-1)!}\lim_{s\rightarrow -a} \frac{d^{q-1}}{ds^{q-1}}\left\{\Phi(s;x)\Li_{k_1,\ldots,k_r}(x_1,\ldots,x_r;1+a_1+s,\ldots,1+a_r+s)\right\}\nonumber\\
&=\sum_{j+m=q-1,\atop j,m\geq 0}\sum_{|\overrightarrow{\bfm}_{1,r}|=m,\atop \forall\ m_l\geq 0} B_m\left(\overrightarrow{\bfk}_{1,r};\overrightarrow{\bfm}_{1,r}\right) \widehat{\Li}_{j+1}\left(x^{-1};a\right)\Li_{\overrightarrow{(\bfk+\bfm)}_{1,r}}\left(\overrightarrow{\bfx}_{1,r};\overrightarrow{(1-a+\bfa)}_{1,r}\right).
\end{align}
For a pole of order $k_j$ at $-n - a_j$ (where $n \in \mathbb{N}$), by applying \eqref{equ-extend-xuzhou-two} and \eqref{Laurent-expansion-Cmhpolyf} and performing direct computation, the residue is obtained as
\begin{align}\label{equ-proof-thm-two-residu-4}
&{\rm Res}[G_q(s),-n-a_j]=\lim_{s\rightarrow -n-a_j} \frac{d^{k_j-1}}{ds^{k_j-1}}\left\{\frac{(s+n+a_j)^{k_j}}{(k_j-1)!}\frac{\Phi(s;x)\Li_{\overrightarrow{\bfk}_{1,r}}(\overrightarrow{\bfx}_{1,r};\overrightarrow{(1+\bfa+s)}_{1,r})}{(s+a)^q}\right\}\nonumber\\
&=\sum_{0\leq i+m\leq k_j-1,\atop i,m\geq 0}\sum_{|\overrightarrow{\bfm}_{1,r}|-m_j=m,\atop \forall\ m_l\geq 0}D_{q,m,i,j}\left(\overrightarrow{\bfk}_{1,r};\overrightarrow{\bfm}_{1,r}\right)
\Li_{\overrightarrow{(\bfk+\bfm)}_{j+1,r}}\left(\overrightarrow{\bfx}_{j+1,r};\overrightarrow{(1-a_j+\bfa)}_{j+1,r}\right)\nonumber\\
&\qquad\qquad\qquad\qquad\times \widehat{\Li}_{i+1}\left(x^{-1};a_j\right)\frac{\zeta_{n-1}\left(\overleftarrow{(\bfk+\bfm)}_{1,j-1};\overleftarrow{\bfx}^{-1}_{1,j-1};\overleftarrow{(a_j-\bfa)}_{1,j-1}\right)}
{(n+a_j-a)^{q+k_j-i-m-1}}(\widetilde{\bfx}_{0,r})^n.
\end{align}
By the residue theorem stated in Lemma \ref{lem-redisue-thm}, the sum of all residues must equal zero, i.e.,
\begin{align}\label{equ-proof-thm-two-residu-5}
&\sum_{n=0}^\infty {\rm Res}[G_q(s),n]+\sum_{n=1}^\infty {\rm Res}[G_q(s),-n]\nonumber\\
&\quad+\sum_{n=1}^\infty \sum_{j=1}^r {\rm Res}[G_q(s),-n-a_j]+{\rm Res}[G_q(s),-a]=0.
\end{align}
Substituting the four residue results \eqref{equ-proof-thm-two-residu-1}-\eqref{equ-proof-thm-two-residu-4} into \eqref{equ-proof-thm-two-residu-5} and replacing $x$ with $\widetilde{\bfx_{0,r}^{-1}}$ in the resulting expression completes the proof of Theorem \ref{mainthm-CMHZVs}. \hfill$\square$

Setting $r=1,k_1=p,x_1=y$ and $a_1=b$ in Theorem \ref{mainthm-CMHZVs} yields the following corollary, namely the symmetry formula for cyclotomic double Hurwitz zeta values..
\begin{cor}\label{corone-sect-4} Let $x,y$ be roots of unity. For $a,b\in \mathbb{C}\setminus \Z$ and $(q,x),(p,y)\neq(1,1)$ with $b-a\notin \N^-$, we have
\begin{align*}
&(xy)^{-1} \Li_{q,p}(x,y;a,b)+(-1)^q y^{-1} \Li_p(y;b)\Li_q\Big(x^{-1};1-a\Big)\\
&+(-1)^{p+q} \Li_{p,q}\Big(y^{-1},x^{-1};1-b,1-a\Big)\\
&+\sum_{j+m=q-1,\atop j,m\geq 0} (-1)^m \binom{m+p-1}{p-1} \Li_{p+m}(y;1+b-a)\\&\qquad\qquad\times\left((-1)^j \Li_{j+1}\Big((xy)^{-1};1-a\Big)-(xy)^{-1}\Li_{j+1}(xy;a) \right)\\
&+(-1)^q \sum_{i=0}^{p-1} \binom{p+q-i-2}{q-1} \Li_{p+q-i-1}\Big(x^{-1};1+b-a\Big)\\&\qquad\qquad\times\left((-1)^i \Li_{i+1}\Big((xy)^{-1};1-b\Big)-(xy)^{-1}\Li_{i+1}(xy;b) \right)=0.
\end{align*}
\end{cor}

From Corollary \ref{corone-sect-4} with the help of \eqref{equ-thm-main-two}, the following cases can be easily provided.
\begin{exa} Let $x$ be root of unity. For $a,b\in \mathbb{C}\setminus \Z$ and $b-a\notin \N^-$, we have
\begin{align*}
&\Li_{2,2}\Big(x^{-1},x;a,b\Big)+\Li_{2,2}\Big(x^{-1},x;1-b,1-a\Big)\\&\qquad+x^{-1} \Li_{2}(x;b)\Li_2(x;1-a)\\
&=\pi^2\Big(\csc^2(\pi a)+\csc^2(\pi b)\Big)\Li_2(x;1+b-a)\\&\qquad+2\pi \Big(\cot(\pi b)-\cot(\pi a)\Big)\Li_3(x;1+b-a),
\end{align*}
and
\begin{align*}
&\Li_{2,2}\Big(x^{-1},-x;1-b,1-a\Big)-\Li_{2,2}\Big(-x^{-1},x;a,b\Big)\\&\qquad+x^{-1} \Li_{2}(x;b)\Li_2(-x;1-a)\\
&=\pi^2 \csc(\pi a)\cot(\pi a) \Li_2(x;1+b-a)+\pi^2 \csc(\pi b)\cot(\pi b)\Li_2(-x;1+b-a)\\
&\quad+2\pi \csc(\pi b) \Li_3(-x;1+b-a)-2\pi \csc(\pi a) \Li_3(x;1+b-a).
\end{align*}
\end{exa}

It is worth emphasizing that in the collaborative paper \cite{LiXu2026} by the present author with Li, it was shown that the results obtained via contour integration can be regularized. The regularization process is not addressed in this paper; interested readers are encouraged to explore this aspect.

Finally, we conclude this paper by raising an open question.
\begin{qu}
Can the method of contour integration employed in this paper be extended to study other variants of multiple zeta values, such as the Schur multiple zeta values, Mordell-Tornheim zeta functions, Barnes multiple zeta functions, and so forth? The author believes that the key lies in obtaining specific forms of Laurent expansions and Taylor expansions for these variants of multiple zeta values.
\end{qu}

{\bf Declaration of competing interest.}
The author declares that he has no known competing financial interests or personal relationships that could have
appeared to influence the work reported in this paper.

{\bf Data availability.}
No data was used for the research described in the article.

{\bf Acknowledgments.} The author would like to express sincere gratitude to Dr. Steven Charlton and Jia Li for their valuable comments and discussions on this paper, and also thanks the anonymous reviewers for their valuable suggestions, which helped improve the quality of this paper. The author is supported by the General Program of Natural Science Foundation of Anhui Province (Grant No. 2508085MA014).

\end{document}